\documentclass[12pt]{article}

\usepackage{times}

\usepackage[a4paper,text={128mm,185mm}, centering]{geometry}

\usepackage{titlesec}

\titleformat{\section}[hang]%
{\bfseries\large}{\thesection.}{1ex}{}%

\titleformat{\subsection}[hang]%
{\bfseries}{\thesubsection}{1ex}{}%

\usepackage{enumerate,xspace}
\usepackage{amsmath,xspace,amssymb}
\usepackage[dvips]{graphics}
\usepackage{proof}
\usepackage{latexsym}  
\usepackage{euscript}  

\input xy
\xyoption{all}
\xyoption{2cell}
\UseAllTwocells
\title{Differential bundles and fibrations \\ for tangent categories}
\author{J.R.B. Cockett\thanks{Research supported by an NSERC Discovery grant and the Centre International de Recontres Math\'{e}matiques.} 
        \ and G.S.H. Cruttwell\thanks{Research supported by an NSERC Discovery grant.}
\\ Department of Computer Science, \\
University of Calgary, Alberta, Canada and \\
Department of Mathematics and Computer Science, \\ 
Mount Allison University, Sackville, Canada.}

\newtheorem{observation}{Remark}[section]
\newtheorem{lemma}[observation]{Lemma}  
\newtheorem{theorem}[observation]{Theorem}
\newtheorem{definition}[observation]{Definition}
\newtheorem{example}[observation]{Example}
\newtheorem{remark}[observation]{Remark}

\newtheorem{proposition}[observation]{Proposition} 
\newtheorem{corollary}[observation]{Corollary} 
 
\newcommand{\proof}{\noindent{\sc Proof:}\xspace}
\def\endproof{~\hfill$\Box$\vskip 10pt}

\newcommand{\x}{\times}
\newcommand{\<}{\langle}
\renewcommand{\>}{\rangle}

\newcommand{\cd}[1]{{\bf [CD.{#1}]}}

\newcommand{\B}{\ensuremath{\mathbb B}\xspace}
\newcommand{\C}{\ensuremath{\mathbb C}\xspace}
\newcommand{\E}{\ensuremath{\mathbb E}\xspace}

\newcommand{\N}{\ensuremath{\mathbb N}\xspace}

\newcommand{\T}{\ensuremath{\mathbb T}\xspace}
\newcommand{\X}{\ensuremath{\mathbb X}\xspace}
\newcommand{\Y}{\ensuremath{\mathbb Y}\xspace}

\newcommand{\op}{\mbox{\scriptsize op}}

\newcommand{\p}{\pi}
\newcommand{\mb}[1]{\mathbf{#1}}
\newcommand{\act}{\triangleright}

\newcommand\nats{\hbox{$I \kern - .38em N$}} 
\newcommand\ints{\hbox{$Z \kern - .65em Z$}} 

\makeatletter

\newdimen\w@dth

\def\setw@dth#1#2{\setbox\z@\hbox{\scriptsize $#1$}\w@dth=\wd\z@
\setbox\@ne\hbox{\scriptsize $#2$}\ifnum\w@dth<\wd\@ne \w@dth=\wd\@ne \fi
\advance\w@dth by 1.2em}

\def\t@^#1_#2{\allowbreak\def\n@one{#1}\def\n@two{#2}\mathrel
{\setw@dth{#1}{#2}
\mathop{\hbox to \w@dth{\rightarrowfill}}\limits
\ifx\n@one\empty\else ^{\box\z@}\fi
\ifx\n@two\empty\else _{\box\@ne}\fi}}
\def\t@@^#1{\@ifnextchar_ {\t@^{#1}}{\t@^{#1}_{}}}

\def\t@left^#1_#2{\def\n@one{#1}\def\n@two{#2}\mathrel{\setw@dth{#1}{#2}
\mathop{\hbox to \w@dth{\leftarrowfill}}\limits
\ifx\n@one\empty\else ^{\box\z@}\fi
\ifx\n@two\empty\else _{\box\@ne}\fi}}
\def\t@@left^#1{\@ifnextchar_ {\t@left^{#1}}{\t@left^{#1}_{}}}

\def\two@^#1_#2{\def\n@one{#1}\def\n@two{#2}\mathrel{\setw@dth{#1}{#2}
\mathop{\vcenter{\hbox to \w@dth{\rightarrowfill}\kern-1.7ex
                 \hbox to \w@dth{\rightarrowfill}}%
       }\limits
\ifx\n@one\empty\else ^{\box\z@}\fi
\ifx\n@two\empty\else _{\box\@ne}\fi}}
\def\tw@@^#1{\@ifnextchar_ {\two@^{#1}}{\two@^{#1}_{}}}

\def\tofr@^#1_#2{\def\n@one{#1}\def\n@two{#2}\mathrel{\setw@dth{#1}{#2}
\mathop{\vcenter{\hbox to \w@dth{\rightarrowfill}\kern-1.7ex
                 \hbox to \w@dth{\leftarrowfill}}%
       }\limits
\ifx\n@one\empty\else ^{\box\z@}\fi
\ifx\n@two\empty\else _{\box\@ne}\fi}}
\def\t@fr@^#1{\@ifnextchar_ {\tofr@^{#1}}{\tofr@^{#1}_{}}}

\newdimen\W@dth
\def\setW@dth#1#2{\setbox\z@\hbox{$#1$}\W@dth=\wd\z@
\setbox\@ne\hbox{$#2$}\ifnum\W@dth<\wd\@ne \W@dth=\wd\@ne \fi
\advance\W@dth by 1.2em}

\def\T@^#1_#2{\allowbreak\def\N@one{#1}\def\N@two{#2}\mathrel
{\setW@dth{#1}{#2}
\mathop{\hbox to \W@dth{\rightarrowfill}}\limits
\ifx\N@one\empty\else ^{\box\z@}\fi
\ifx\N@two\empty\else _{\box\@ne}\fi}}
\def\T@@^#1{\@ifnextchar_ {\T@^{#1}}{\T@^{#1}_{}}}

\def\T@left^#1_#2{\def\N@one{#1}\def\N@two{#2}\mathrel{\setW@dth{#1}{#2}
\mathop{\hbox to \W@dth{\leftarrowfill}}\limits
\ifx\N@one\empty\else ^{\box\z@}\fi
\ifx\N@two\empty\else _{\box\@ne}\fi}}
\def\T@@left^#1{\@ifnextchar_ {\T@left^{#1}}{\T@left^{#1}_{}}}

\def\Tofr@^#1_#2{\def\N@one{#1}\def\N@two{#2}\mathrel{\setW@dth{#1}{#2}
\mathop{\vcenter{\hbox to \W@dth{\rightarrowfill}\kern-1.7ex
                 \hbox to \W@dth{\leftarrowfill}}%
       }\limits
\ifx\N@one\empty\else ^{\box\z@}\fi
\ifx\N@two\empty\else _{\box\@ne}\fi}}
\def\T@fr@^#1{\@ifnextchar_ {\Tofr@^{#1}}{\Tofr@^{#1}_{}}}

\def\Two@^#1_#2{\def\N@one{#1}\def\N@two{#2}\mathrel{\setW@dth{#1}{#2}
\mathop{\vcenter{\hbox to \W@dth{\rightarrowfill}\kern-1.7ex
                 \hbox to \W@dth{\rightarrowfill}}%
       }\limits
\ifx\N@one\empty\else ^{\box\z@}\fi
\ifx\N@two\empty\else _{\box\@ne}\fi}}
\def\Tw@@^#1{\@ifnextchar_ {\Two@^{#1}}{\Two@^{#1}_{}}}

\def\to{\@ifnextchar^ {\t@@}{\t@@^{}}}
\def\from{\@ifnextchar^ {\t@@left}{\t@@left^{}}}
\def\tofro{\@ifnextchar^ {\t@fr@}{\t@fr@^{}}}
\def\To{\@ifnextchar^ {\T@@}{\T@@^{}}}
\def\From{\@ifnextchar^ {\T@@left}{\T@@left^{}}}
\def\Two{\@ifnextchar^ {\Tw@@}{\Tw@@^{}}}
\def\Tofro{\@ifnextchar^ {\T@fr@}{\T@fr@^{}}}

\makeatother

\input{diagxy}

\begin{document}
\maketitle
\begin{abstract}
Tangent categories are categories equipped with a tangent functor: an endofunctor with certain natural  transformations which make it behave like the tangent bundle functor on the category of smooth manifolds.  They provide an abstract setting for differential geometry by axiomatizing key aspects of the subject which allow the basic theory of these geometric settings to be captured.  Importantly, they have models not only in classical differential geometry and its extensions, but also in algebraic geometry, combinatorics, computer science, and physics.  

This paper develops the theory of ``differential bundles'' for such categories, considers their relation to ``differential objects'', and develops the theory of fibrations of tangent categories.  Differential bundles generalize the notion of smooth vector bundles in classical differential geometry.  However, the definition departs from the standard one in several significant ways: in general, there is no scalar multiplication in the fibres of these bundles, and in general these bundles need not be locally trivial.  

To understand how these differential bundles relate to differential objects, which are the generalization of vector spaces in smooth manifolds, requires some careful 
handling of the behaviour of pullbacks with respect to the tangent functor.   This is captured by ``transverse'' and ``display'' systems for 
tangent categories, which leads one into the fibrational theory of tangent categories.  A key example of a tangent fibration is provided by the ``display'' differential bundles 
of a tangent category with a display system. Strikingly, in such examples the fibres are Cartesian differential categories demonstrating a -- not unexpected -- tight 
connection between the theory of these categories and that of tangent categories.
\end{abstract}

\begin{minipage}{118mm}{\small

}\end{minipage}






\tableofcontents 


\section{Introduction}

Tangent categories provide an axiomatic setting for abstract differential geometry.  They were first introduced in \cite{rosicky} as a category equipped with a ``tangent functor'' which associated to each object an Abelian group bundle with additional structure.  It was shown in that paper that tangent categories encompass both standard differential geometry settings \cite{chern-chen-lam}, algebraic geometry settings, and settings arising in synthetic differential geometry \cite{kock}.  

In \cite{sman3}, the present authors slightly generalized this notion so that the bundles were only assumed to be commutative monoids.  This allowed for a key source of new examples of tangent categories arising from computer science and combinatorics.  In computer science, the resource $\lambda$-calculus \cite{curien} and the differential $\lambda$-calculus \cite{diffLambda}  were developed in parallel.  They were, eventually, unified \cite{manzonetto} as being calculii with their semantics in Cartesian differential categories \cite{cartDiff}.  In combinatorics the differential of a combinatorial species \cite{bergeron} is also an important tool.  This idea was developed more abstractly into a differential of polynomial functors which were then connected to the differential of datatypes (see \cite{polyFunctors}, \cite{abbott-thesis}, and \cite{abbott}).     These provide examples of settings in which there is a notion of differentiation but in which negation has no natural meaning.  An important aspect of this paper is to spell out in detail the connection between Cartesian differential categories, which provide a unifying framework for the settings above, and tangent categories.  As we shall see, one aspect of this connection hinges on the notion of a differential bundle.

A fundamental structure in differential geometry is the (smooth) vector bundle.  Vector bundles are important in differential geometry because they algebraically capture the notion of local coordinate systems for manifolds.  They thus provide a way to describe additional structure on manifolds such as vector fields, symplectic forms, and differential forms. The tangent bundle of a manifold is, of course, itself a vector bundle, but one can also form the product (or the ``Whitney sum'') of two vector bundles and their tensor product. Significantly, one can pullback vector bundles along smooth maps allowing the transport of this local structure.  One can also apply the tangent bundle functor to a vector bundle to produce another vector bundle.   

In the abstract setting of a tangent category, it is not immediately obvious how to define vector bundles as, in particular, there is no assumption of any sort of  ``object of real numbers'' from which one can define vector spaces with a scalar multiplication.  However, quoting from the Wikipedia entry on vector bundles\footnote{As of Nov. 2nd, 2016.}, ``smooth vector bundles have a very important property not shared by more general fibre bundles.  Namely, the tangent space $T_v(E_x)$ at any $v \in E_x$ can be naturally identified with the fibre $E_x$ itself.  This identification is obtained through the vertical lift...''.  It is this key structure which we use as the \emph{definition} of the generalization of smooth vector bundle to arbitrary tangent categories.  That is, rather than ask that each fibre of a map $q: E \to M$ be a vector space (smoothly), we ask that $q$ be an additive bundle with, in addition, a ``lift'' map
\[ \lambda: E \to T(E) \]
which enjoys certain properties so that ``the tangent space $T_v(E_x)$ at any $v \in E_x$ can be naturally identified with the fibre $E_x$ itself''.  Because at this level of generality there is no scalar multiplication -- and in order to emphasize the connections we establish in this paper -- we call such bundles {\em differential bundles\/}.  

Differential bundles, so defined, enjoy all of the key properties of ordinary smooth vector bundles.  Namely that the tangent bundle is a differential bundle (example \ref{basicdiffbundles}), applying the tangent bundle functor to a differential bundle produces another differential bundle (corollary \ref{lemmaTDiffBundle}), and the pullback of a  differential bundle along any map is again a differential bundle (lemma \ref{pullbackDiffBundle}).  Moreover, we show that the obvious maps between such bundles, namely those that preserve the ``lift'' operation, suitably generalize ordinary linear maps between vector spaces.  This observation thus gives an alternative perspective on the meaning of ``linearity'': linearity can be seen as the preservation of the lift map, rather than the preservation of an action\footnote{In Synthetic Differential Geometry linear maps are often referred to as {\em homogeneous\/} maps: intuitively they preserve the multiplicative ``action'' of the infinitesimal object on differential bundles. In tangent categories this ``action'' manifests itself in a dual form as a functorial ``coaction'' of the tangent bundle functor on the bundle -- provided by the lift map.} by real number objects.  

With the basic definition of differential bundles and their properties, one can look at connections on such bundles, a topic which will be treated in a future paper by the authors.    

However, there is more to say about these differential bundles as objects of interest in their own right.   The authors' previous paper on tangent categories \cite{sman3} defined the analog of vector spaces in a tangent 
category, calling these objects differential objects.  Just as vector bundles are vector spaces in a slice category, it is then natural to ask whether differential bundles are differential objects in a slice tangent category.  

Here, though, lies a difficulty which is at the heart of much of the rest of this paper.  In \cite{rosicky} Rosick\'y had proposed that the slice of a tangent 
category should again be a tangent category with respect to the ``vertical'' tangent bundle.  Furthermore, he had suggested that, for this to be so, it sufficed 
that the construction of the vertical bundle (given by pulling back over the zero of the tangent bundle) needed to exist and be preserved by the tangent functor.   
Of course, to be a tangent category there are other pullbacks which need to be present.  In particular, the pullback of the projection from the
tangent bundle functor along itself -- called here, and in the authors' previous paper, $T_2$ -- needs to be present and preserved.  While the condition 
Rosick\'y had suggested was clearly necessary, it did not allow -- as far as we could see -- a construction of this pullback.   This led us to  
the view that the question of what pullbacks should exist (and be preserved), was more subtle than had been supposed and deserved 
careful treatment.

In a tangent category it is certainly not the case that arbitrary pullbacks exist let alone that the tangent functor preserves those pullbacks.  Recall, for 
example, even in classical differential geometry the lack of pullbacks has given rise to a detailed theory -- the theory of transverse maps and 
submersions --  of when such pullbacks exist and are preserved.   For tangent categories, therefore, it should be no surprise that a commensurate 
theory needs to be developed.  Thus, in particular, for differential bundles to become differential objects in the slice, it should be expected that  some
-- somewhat subtle -- conditions on the pullbacks may arise.    

To deal with these issues in tangent categories it is necessary to have explicit structural descriptions of which pullbacks must exist and be preserved.  Toward this end we introduce two notions: {\em transverse systems} and {\em display systems}.  A transverse system on a category specifies a class of pullbacks in the category subject to a small set of axioms.  Every tangent category comes equipped with a minimal transverse system which comprises the pullbacks mandated by the tangent structure to exist and be preserved.  However, it is quite possible that a given tangent category have a larger transverse system: in the category of smooth manifolds, there is a natural transverse system, first defined in \cite{rene-thom}, consisting of all pullbacks of a pair of maps which are transverse (in the sense of differential geometry) to each other. In addition to this, it is also often useful to specify a class of maps in the category along which all pullbacks exist and are in the transverse system.  We call these \emph{display} maps, as they are closely related to the display systems introduced by Paul Taylor \cite{paulTaylorBook} in the study of fibrations.  A tangent category with a compatible display system is called a {\em display\/} tangent category.

With these ideas in hand, we can look at the subject of differential bundles and slice tangent categories from the prospective of fibrations.   A tangent fibration is then a fibration of tangent categories satisfying certain additional axioms.  A key result for these structures is that each fibre of a tangent fibration is again a tangent category (theorem \ref{tangent-substitution}).  Moreover, if we have a display tangent category $\X$, then the subcategory of the arrow category of $\X$ consisting of display maps is a tangent fibration.  Thus, in a display tangent category, not only is the category of display maps over a fixed object a tangent category, but also the differential objects in this slice tangent category are exactly differential bundles whose projections are display maps (\ref{displayedDiffBundleEquivalence}).   Another important example of a tangent fibration, which captures the idea of partial derivatives, is given in  \ref{tanFibrationExamples} (a).  

Bringing the two main ideas of this paper together, for a display tangent category we can consider the fibration of differential bundles (with projections display maps).  Each fibre of this fibration has the very special property that every object is a differential object in a canonical way: this makes each fiber a Cartesian differential category.  This then establishes the tight relationship between differential bundles in a tangent category and Cartesian differential categories.  

\medskip
\noindent
{\bf Organization}

\medskip
 In section \ref{secBasicTangent}, we recall the definition of a tangent category, introduce differential bundles in tangent categories, and study their properties.  In section \ref{diff-structure} we begin the study of how differential bundles relate to differential objects and Cartesian differential categories.   We recall that a Cartesian differential category is always a tangent category in which every object is a differential object.  However, for a  tangent category to be a Cartesian differential category not only must every object be a differential object but also there must be a global coherence between these structures.  The key concept in this section is, thus, the notion of a tangent category with coherent differential structure.  This idea becomes central when we later study the tangent category of differential bundles over a fixed base as these categories naturally have coherent differential structure.  In section \ref{secTransAndDisplay} we discuss the notions of transverse and display systems for arbitrary categories and for tangent categories.  In this section we also revisit the definitions of morphisms of tangent categories and differential bundles when the tangent categories have transverse and display systems and we show how some earlier results can be obtained more conceptually from these notions.  Finally, in section \ref{secTanFibrations}, we define and study fibrations between tangent categories, eventually concluding with the result that there is a tangent fibration of display differential bundles; in this fibration, every fibre is a tangent category with coherent differential structure and, thus, is a differential fibration.

\section{Basic tangent categories and differential bundles}\label{secBasicTangent}

We shall begin with the definition of tangent category as a category with tangent structure consisting of a functor, some natural transformations, and certain pullbacks.  The basic idea was introduced in \cite{rosicky}.  In \cite{sman3} those ideas were generalized slightly by allowing the  tangent bundle to be a commutative monoid rather than an Abelian group bundle.  This generalization was motivated by examples arising from computer science and combinatorics in which negation has no natural interpretation.  

\subsection{Tangent categories}

Throughout this paper, following \cite{cartDiff} and \cite{sman3}, we write composition in diagrammatic order, so that $f$, followed by $g$, is written as $fg$.  

If $M$ is an object in a category $\X$ an \textbf{additive bundle over $M$}, $q: E \to M$, consists of a map $q$ which admits finite pullback powers along itself 
$$\xymatrix{& E \ar[dr]^q \\ E_n \ar@{..>}[rd]_{\pi_{n-1}} \ar@{..>}[ru]^{\pi_0} & \vdots & M \\ & E \ar[ru]_q }$$ 
which is a commutative monoid in the slice category over $M$, $\X/M$.  In particular this means there is an addition operation, which we shall often write as $\sigma: E_2 \to E$ and must satisfy the usual 
requirements of commutativity and associativity, and a unit for this addition, which we shall often write as $\zeta:M \to E$.  A map between such bundles will, in general, just be a commutative square
$$\xymatrix{E \ar[rr]^e \ar[d]_q && E' \ar[d]^{q'} \\ B \ar[rr]_b && B'}$$
written $(e,b): q \to q'$.   If, in addition, such a map of bundles preserves the addition -- that is $e_2 \sigma' = \sigma e$ and $b \zeta' = \zeta e$ -- then we shall say that $(e,b)$ is an {\bf additive bundle morphism}.

\begin{definition}\label{defnTangentCategory}
For a category $\X$, tangent structure $\T = (T,p,0,+,\ell,c)$ on $\X$ consists of the following data:
\begin{itemize} 
	\item (\textbf{tangent functor})  a functor $T: \X \to \X$ with a natural transformation $p: T \to I_{\X}$ such that each 
	$p_M: T(M) \to M$ admits finite pullback powers along itself which are preserved by each $T^n$;
	\item (\textbf{additive bundle}) natural transformations $+:T_2 \to T$ (where $T_2$ is the pullback of $p$ over itself) and $0: I \to T$ making each $p_M: TM \to M$ an additive bundle;
	\item (\textbf{vertical lift}) a natural transformation $\ell:T \to T^2$ such that for each $M$
	$$(\ell_M,0_M): (p: TM \to M,+,0) \to (Tp: T^2M \to TM,T(+),T(0))$$ 
	is an additive bundle morphism;
        \item (\textbf{canonical flip}) a natural transformation $c:T^2 \to T^2$ such 
               that for each $M$
               $$(c_M,1): (Tp: T^2M \to TM, T(+),T(0)) \to (p_T: T^2M \to TM,+_T,0_T)$$ 
               is an additive bundle morphism;
        \item (\textbf{coherence of $\ell$ and $c$}) $c^2 = 1$ (so $c$ is an isomorphism), $\ell c = \ell$, and the following diagrams commute:
\[
\xymatrix{T \ar[r]^{\ell} \ar[d]_{\ell} &T^2 \ar[d]^{T(\ell)} \\
                 T^2 \ar[r]_{\ell_T} & T^3}
~~~
\xymatrix{T^3  \ar[r]^{T(c)} \ar[d]_{c_T} & T^3 \ar[r]^{c_T} & T^3 \ar[d]^{T(c)} \\
                  T^3 \ar[r]_{T(c)} & T^3 \ar[r]_{c_T} & T^3}
~~~
\xymatrix{T^2 \ar[d]_{c} \ar[r]^{\ell_T} & T^3 \ar[r]^{T(c)} & T^3 \ar[d]^{c_T} \\
                  T^2 \ar[rr]_{T(\ell)} & & T^3}
\]
	\item (\textbf{universality of vertical lift}) defining $v: T_2M \to T^2M$ by $v := \<\p_0\ell, \p_10_T\>T(+)$,
	the following diagram is a pullback\footnote{In \cite{sman3} this 
condition is given as the requirement that $v$ is the equalizer of $T(p)$ and $pp0$: this followed the approach in \cite{rosicky}. However, we 
now believe that the condition is more naturally expressed as a pullback and this view gives a smoother development of the theory.  
The equivalence to the equalizer requirement is given (in the more general context of differential bundles) in Lemma \ref{lemmaVerticalLift} below.} which 
is preserved by each $T^n$: 
$$\xymatrix{T_2(M) \ar[d]_{\pi_0p=\pi_1p} \ar[rr]^{v} &  & T^2(M) \ar[d]^{T(p)} \\  M \ar[rr]_{0} &  & T(M)}$$	
\end{itemize}
A category with tangent structure, $(\X,\T)$, is a {\bf tangent category}.   A tangent category is said to be {\bf Cartesian\/}  if it has finite products which are preserved by $T$. 
\end{definition}

The requirement that each $T^n$ preserve the pullback expressing the universality of the vertical lift is, in fact, a consequence of the other requirements.  This is because
the canonical flip can be used to transform a universality diagram acted on by $T^n$ back into a ``top-level'' universality diagram by flipping the maps up to the top level 
(see \cite{sman3} Lemma 2.15 for the preservation of the equalizer form of the condition).  Here we wish to emphasize the fact that $T^n$ should preserve these pullbacks in order 
to hint at a more general pattern.

One can think of the vertical lift as a comultiplication for the tangent functor and the flip as a symmetry transformation.  With this perspective $\ell c = \ell$ asserts that the comultiplication is 
commutative and $\ell T(\ell) = \ell \ell$ asserts that the comultiplication is coassociative.  One might expect that as the vertical lift acts as a comultiplication that the tangent functor should be a 
comonad, however, significantly, this is {\em not\/} the case.   It is the case, however, that $(T,\eta.\mu)$, where $\eta_m := 0_M: M \to T(M)$ and $\mu := \< p,T(p)\> +_M: T^2(M) \to T(M)$, 
is a monad (see \cite{sman3}, section 3.2).  

Tangent categories provide a strict generalization of traditional settings for differential geometry in a number of respects.  An immediate and striking difference is that in tangent categories the tangent bundles are assumed only to be commutative monoids -- that is they may lack negatives.  Many of the results of classical differential geometry require that one has negatives.  It should, therefore, be emphasized that, having negatives is a property (rather than structure) and {\em all\/} the basic results in this paper are, of course, true when one does have negation. Thus, importantly, this work {\em strictly\/} includes the settings of classical differential geometry.  Another important difference (in contrast to settings like synthetic differential geometry) is that the definition assumes no ``object of real numbers''.  

\begin{example} \label{tangent-category-examples}
{\em We briefly list several important examples of tangent categories:
\begin{enumerate}[{\em (i)}]
\item Finite dimensional smooth manifolds with the usual tangent bundle structure.  For tangent vectors $u$ and $v$ at a point $x$, the vertical lift is given by
	\[ \frac{d}{dt}|_{t = 0} (u + tv) \]
(see \cite{natural}, pg. 55).  For smooth manifolds it is well known that the pullback of two maps $f$ and $g$ exists and is preserved by the tangent functor whenever the maps are {\em transverse\/}.  This means that, at each point $f(x)=g(y)$, the tangent space is generated by image of tangent spaces under $T(f)$ and $T(g)$.   A {\em submersion\/} is a map $f$ which is surjective on the tangent spaces at each $f(x)$: these are clearly transverse to every map.  In particular, the projection $p_M: T(M) \to M$ is a submersion (eg.,  see \cite[7.1.e]{lee}) so that pullbacks along it exist and are preserved by the tangent functor.  
\item Cartesian differential categories \cite{cartDiff} are tangent categories, with $T(A) = A \times A$, $T(f) = \<Df,\pi_1f\>$, $p = \pi_1$ \cite{sman3}, and the vertical lift $\ell$ and canonical flip $c$ given by
	\[ \ell(u,x) = (u,0,0,x) ~\mbox{and}~ c(u,v,w,x) = (u,w,v,x). \]
In any Cartesian category one always has pullbacks along projections.  Furthermore,  the tangent bundle functor $T$ preserves products and projections as 
$$T(\pi_1) =   \<D\pi_1,\pi_1\pi_1\> =  \< \pi_0\pi_1,\pi_1\pi_1\> = \pi_1 \x \pi_1$$
and the latter is a projection up to equivalence.  Clearly, $T^n$ also preserves these pullbacks.

It should be noted that many important examples of Cartesian differential categories (particularly those arising from Computer Science) do not have negatives.   For specific examples, see \cite{diffCats}, section 2.5: by \cite{cartDiff} section 3.2, the coKleisli categories of each of these (``monoidal'') differential categories form Cartesian differential categories and hence also form tangent categories.  
\item  Any category is trivially a tangent category by setting the tangent functor to be the identity functor and $p$, $0$, $+$, $c$, and $\ell$ all to be the identity natural transformation.  The fact that every category is trivially a tangent category reminds one that tangent structure is certainly {\em structure\/}, rather than a property of the underlying category.  
\item The infinitesimally linear objects in any model of synthetic differential geometry \cite{kock} gives an example of {\em representable\/} tangent structure (see \cite{sman3}, section 5.2 for a characterization of when the tangent functor $T$ in a tangent category is representable).   If $D$ is the object of infinitesimals, then we take $TM = M^D$. The vertical lift is then given by exponentiating the multiplication map $D \times D \to D$.  In models of SDG all pullbacks exist and, as $(\_)^D$ is a right adjoint, it preserves all pullbacks. Thus, in these settings {\em every\/}  map has a pullback which is preserved by the tangent bundle functor and its powers.

\item The opposite of the category of finitely presented commutative rings (or more generally rigs) is another standard example of a category with representable 
tangent structure: here $D$ is the ``rig of infinitesimals'', $\N[\varepsilon] := \N[x]/(x^2=0)$.  
\item A source of examples, from \cite{rosicky}, uses the fact that if $(\X, \T)$ is a tangent category then the functors from $\X$ to $\textbf{set}$ which preserve both the wide pullbacks of $T^n(p)$, and the pullback from the universality of the lift with natural transformations as arrows forms a tangent category.    The tangent functor $T^*$ is then given by $T^{*}(F) := TF$.  In fact, this works for any category $\Y$ in place of $\textbf{set}$ and functors $\X \to \Y$ which preserve the required pullbacks.  This source of examples includes, for example, $C^\infty$-rings (see \cite{reyes} chapter 1) and more generally the product preserving functors from any Cartesian differential category.  Viewing a Cartesian differential category as a (generalized) many-sorted theory of differentiation, it is pleasing to know that its category of algebras (in any category with sufficient limits) is necessarily a tangent category.
\item The category of functors, ${\sf Cat}(\C,\X)$, from any category to a tangent category inherits the tangent structure of $\X$ pointwise.  Thus, for example the category of arrows in a tangent category, $\X^{\mathbf 2}$, is a tangent category with $T(A \to^f B) = T(A) \to^{T(f)} T(B)$.
\item Convenient manifolds with the kinematic tangent bundle (see \cite{convenient} section 28) form a tangent category, by combining the results of  \cite{convenient-diff} and section 6 of \cite{sman3}.   
\end{enumerate} }
\end{example}

As noted in some of the examples above, the fact that a given category can carry more than one tangent structure implies that being a tangent category is a structure on rather than a property of the category. 

\subsection{Differential bundles}

Vector bundles play an important role in differential geometry and this section defines the corresponding notion for tangent categories.  We call this notion
a {\em differential bundle\/}: it is an additive bundle with, in addition, a {\em lift map\/} satisfying  properties similar to those of the vertical lift for the tangent bundle itself.  The 
morphisms between these bundles will then just be commuting squares between the projections.  However, we will  identify an important subclass of morphisms, called 
the {\em linear\/} morphisms, which must in addition preserve the lift.  Notably, neither ordinary morphisms nor linear morphisms require that the additive structure be preserved.  However,
we shall show that linear maps between differential bundles automatically preserve addition.  

\begin{definition}\label{defnDiffBundles}
A \textbf{differential bundle} in a tangent category consists of an additive bundle on a map $q$ together with a lift map $\lambda$:
$${\sf q} := (q: E \to M, \sigma: E_2 \to E, \zeta: M \to E, \lambda: E \to T(E))$$ 
such that
\begin{itemize}
        \item Finite wide pullbacks of $q$ along itself exist and are preserved by each $T^n$.
        \item $(\lambda,0):(E, q,\sigma,\zeta) \to (T(E), T(q), T(\sigma), T(\zeta))$ is an additive bundle morphism.
	\item $(\lambda,\zeta):(E, q,\sigma,\zeta) \to (T(E), p,+,0)$ is an additive bundle morphism.
	\item The {\bf universality of lift}: 
	the following diagram is a pullback which is preserved by each $T^n$:
	       $$\xymatrix{E_2 \ar[d]_{\p_0q=\pi_1q} \ar[rr]^{\mu} &  & T(E) \ar[d]^{T(q)} \\ M \ar[rr]_{0} & & T(M)}$$
	       where $\mu := \<\p_0\lambda, \p_10\>T(\sigma): E_2 \to T(E)$,
	\item The equation $\lambda \ell = \lambda T(\lambda)$ holds.
\end{itemize}
A {\bf morphism of differential bundles} $(f,g): {\sf q} \to {\sf q'}$ is a pair of maps $f: E \to E'$, $g: M \to M'$ such that $fq' = qg$ (the first diagram below). A morphism of differential bundles is {\bf linear} in case, in addition, it preserves the lift, that is $f \lambda' = \lambda T(f)$ (the second diagram below):
$$\xymatrix{E \ar[d]_q \ar[rr]^f && E' \ar[d]^{q'} \\ M \ar[rr]_{g} && M'} ~~~~~~ 
 \xymatrix{E \ar[d]_\lambda \ar[rr]^f & & E' \ar[d]^{\lambda'} \\ T(E) \ar[rr]_{T(f)} && T(E')}$$
\end{definition}

As discussed in the introduction, differential bundles do not, in general, have a scalar multiplication, nor in general are they ``locally a product''.  For these two reasons we have avoided the term ``vector bundle''.  Instead the key structural property of a differential bundle is the lift which represents the idea that ``the tangent space $T_v(E_x)$ at any $v \in E_x$ can be naturally identified with the fibre $E_x$ itself''.   As we shall see shortly, this provides these bundles with differential structure which motivated the name  {\em differential bundles\/}.  

Clearly the differential bundles and morphisms of a tangent category $\X$ themselves form a category ${\sf DBun}(\X)$ which has an underlying functor 
$$P: {\sf DBun}(\X) \to \X; \begin{array}[c]{c} \xymatrix{ M \ar[r]^q \ar[d]_f & B \ar[d]^b \\ M' \ar[r]_{q'} & B'} \end{array} \mapsto \begin{array}[c]{c} \xymatrix{B \ar[d]^b \\  B'} \end{array}$$
whose fibre over an object $B$ consists of all the bundles with base $B$ with maps which fix the base.  One may hope for $P$ to be some sort of ``tangent fibration'' and, indeed, one of the main objectives of this paper is to make this idea precise.

Some basic examples of differential bundles are:

\begin{example}~ \label{basicdiffbundles}
{\em 
\begin{enumerate}[(i)]
\item Every object has associated with it a ``trivial'' differential bundle ${\sf 1}_M =(1_M,1_M,1_M,0_M)$.  Any differential bundle over $M$ has 
a unique bundle map to this bundle, $(q,1_M): {\sf q} \to {\sf 1}_M$, which is the identity on the base:
$$\xymatrix{E \ar[d]_q \ar[rr]^q & & M \ar[d]^{1_M} \\ M \ar[rr]_{1_M} && M}$$
Furthermore this is a linear map as $\lambda T(q) = q 0$.  Clearly ${\sf 1}_M$ is the final differential bundle in the 
fibre over $M$.  Given any $f: N \to M$ this can also be viewed as a linear morphism between the trivial bundles $(f,f): {\sf 1}_N \to {\sf 1}_M$.
Furthermore, there is always a linear zero bundle morphism $(\zeta,1): {\sf 1}_M \to {\sf q}$.
\item The tangent bundle of each object $M$, ${\sf p}_M = (p: T(M) \to M,+,0,\ell)$, is clearly a differential bundle and by naturality of $\ell$, any map $f: N \to M$ induces a linear map $(T(f),f): {\sf p}_N \to {\sf p}_M$ between these differential bundles.
\item A vector bundle, as defined in \cite{chern-chen-lam}, consists of a smooth map between smooth manifolds $q: E \to B$ which, for each $b \in B$, has a neighbourhood $U_b$  so that $q|_{U_b}$ is smoothly equivalent to the trivial $\pi_0: U_b \x V \to U_b$ for some fixed vector space $V$.   Of course, this does {\em not\/} mean that the whole bundle is trivial (in the sense that $E \simeq B \x  V$): the locally trivial bundles of a vector bundle can glue together in a manner which incorporates twists which give important information about the properties of the space $B$.   

Any vector bundle in smooth manifolds is a differential bundle: see \cite{natural}, pg. 55 for the requisite structure.  A differential bundle in smooth manifolds, however, is not necessarily a vector bundle for a rather simple reason: 
if $B$ is not connected a differential bundle allows different components to have fibres which are vector spaces of different dimensions.   In fact, a differential bundle in smooth manifolds is a vector bundle precisely when all of these fibres happen to 
have the same dimension: see corollary 31 of \cite{submersionsAndBundles}.  
\item In synthetic differential geometry (SDG) the preferred notion of a ``vector bundle'' is a Euclidean module in $\X/M$ (for example, see \cite{bungeConnections}, \cite{kockConnections}, and \cite{lavendhomme}).    Theorem \ref{diffBundlesInSDG}, below, will show that a Euclidean module in a model of SDG is precisely the same as a differential object.  Proposition \ref{diffObjeqdiffBun} will show that in a Cartesian tangent category, a differential object is the same as a differential bundle over the final object.  Finally, in Section \ref{diff-fibrations} we will show that, when the tangent functor preserves sufficient limits, differential bundles over a base $M$ are the same as differential bundles over the final object in the slice category $\X/M$.  In models of SDG, the tangent functor preserves all limits as it is given by exponentiation.     Thus, in a model of SDG, Euclidean modules in $\X/M$ are the same as differential bundles over $M$.   
\end{enumerate}}

\end{example}


\subsection{Two constructions of differential bundles}\label{secConstDiffBundles}

While the above examples show that differential bundles arise frequently in the standard examples, we will like to also show that differential bundles abound in an arbitrary tangent category.    Towards this end we describe two fundamental ways of constructing new differential bundles from existing differential bundles.   

If ${\sf q} = (q,\sigma,\zeta,\lambda)$ is a differential bundle, then define 
$$T({\sf q}):= (T(q),T(\sigma),T(\zeta),T(\lambda)c).$$

\begin{lemma}\label{lemmaTDiffBundle}
If ${\sf q}$ is a differential bundle then so is $T({\sf q})$.
\end{lemma}

\begin{proof}
While it is immediate that this is an additive bundle and that pullbacks of $T(q)$ along itself exist and are preserved by $T^n$,  
it is not immediate that $T(\lambda)c$ acts as a lift.

We start by checking 
$(T(\lambda)c,0_T)\!: T(q) \to T^2(q)$ and $(T(\lambda)c,T(\zeta))\!: T(q) \to p_{T(E)}$ 
are additive bundle 
morphisms.  For the first, as $(\lambda,0): q \to T(q)$ is an additive bundle morphism, applying $T$ we have $(T(\lambda),T(0))\!: T(q) \to T^2(q)$ 
is also.  Composing this with $(c,c): T^2(q) \to T^2(q)$ we obtain $(T(\lambda)c,0): T(q) \to T^2(q)$ as an additive bundle morphism.  
For the second, similarly we apply $T$ to obtain the additive bundle morphism   $(T(q),T(\zeta)): T(q) \to T(p_M)$ but this time compose with $(c,1)$ 
to obtain $(T(\lambda)c,T(\zeta)): T(q) \to p_{T(E)}$ as an additive bundle morphism.  

For the coherence of $T(\lambda)c$, we need to check the following commutes:
$$ \xymatrix{T(E) \ar[d]_{T(\lambda)c} \ar[r]^{T(\lambda)c} & T^2(E) \ar[d]^{\ell} \\ T^2(E) \ar[r]_{T(T(\lambda)c)} &T^3(E) }$$

First note that the last coherence diagram for tangent structure may be re-expressed as 
	\[ T(\ell) c T(c) = c \ell, \]
using the fact that $c^2 = 1$. This gives
\begin{eqnarray*}
T(\lambda)c T(T(\lambda)c) & = & T(\lambda T(\lambda)) c T(c) \\
& = & T(\lambda\ell) c T(c) \\
& = & T(\lambda) T(\ell) c T(c) \\
& = & T(\lambda) c \ell
\end{eqnarray*}
as required.  

For universality of lift we need to show that 
	$$\xymatrix{ T(E_2) \ar[d]_{\p_0T(q)}  \ar[rrr]^{\<\pi_0 T(\lambda)c,\pi_1 0\> T^2(\sigma)} & & & T^2(E) \ar[d]^{T^2(q)} \\
	                        T(M) \ar[rrr]_0 & & & T^2(M)}$$
is a pullback.  We may re-express this square as the outer square of 
	$$\xymatrix{ T(E_2) \ar[d]_{T(\p_0q)}  \ar[rrr]^{\<T(\pi_0 \lambda)c,T(\pi_1 0)\> T^2(\sigma)} & & & T^2(E) \ar[d]^{T^2(q)} \ar[r]^c &T_2(E) \ar[d]^{T^2(q)} \\
	                        T(M) \ar[rrr]_0 & & & T^2(M) \ar[r]_c & T^2(M)}$$
Since $T$ preserves the universality of $q$, the outer square is a pullback which is preserved by $T^n$, as required.  
\end{proof}

Observe also that there are some obvious linear morphisms of bundles associated with this construction:

\begin{corollary}
$(0_E,0_M): {\sf q} \to T({\sf q})$ and $(p_E,p_M): T({\sf q}) \to {\sf q}$ are linear bundle morphisms.
\end{corollary}

Another important way of constructing differential bundles is by pulling back.  The pullbacks involved, however, not only have to exist but they also must be preserved by $T^n$, neither of which is 
guaranteed in an arbitrary tangent category.  The question of which pullbacks exist and are preserved by $T^n$ is a topic to which we will return when we introduce the notion of transverse and display systems.  
In the meantime we shall deal with these matters in a rather ad hoc fashion in order to get this basic construction of differential bundles on the table early.  

If ${\mb q} = (q: E \to M,\zeta,\sigma,\lambda)$ is a differential bundle, and $f: X \to M$ is a map 
then to say the pullback of the bundle along $f$ exists and is preserved by $T^n$ means that the pullback along $f$ of all the wide pullbacks of $q$ along itself must also exist and be preserved by $T^n$.  
It is clear that pulling back ${\sf q}$ along $f$ under these assumptions will provide an additive bundle.  However, 
as the following cube (whose front and back faces are pullbacks) shows this additive bundle also has the data for being a differential bundle as there is a candidate for the lift, $f^{*}(\lambda)$, given by 
$$\xymatrix{ & T(f^{*}(E)) \ar[dd]|{\hole}^<<<<{T(f^{*}(q))} \ar[rr]^{T(f^{*}_E)} & & T(E) \ar[dd]^{T(q)} \\ 
           f^*(E) \ar[dd]_{f^{*}(q)}  \ar@{..>}[ur]^{f^{*}(\lambda)} \ar[rr]^>>>>>>>{f^{*}_E} & & E \ar[dd]^<<<<<<q \ar[ru]_{\lambda} \\
           & T(X) \ar[rr]_<<<<<<{T(f)}|>>>>>>>>>>>>{~~} & & T(M) \\
           X \ar[ur]^{0} \ar[rr]_{f} & & M \ar[ur]_{0}}$$
We shall suggestively call this structure $f^{*}({\sf q})$.  Notice that the cube also immediately shows two further useful facts.   First, as the right face is an additive bundle morphism, the left face must 
also be an additive bundle morphism.  Secondly, if $f^{*}({\sf q})$ is indeed a differential bundle, then there is a linear morphism $f^{*}_{\sf q} = (f^{*}_E,f): f^*({\sf q}) \to {\sf q}$.  
This will clearly be a Cartesian map sitting above $f$ for the functor $P: {\sf DBun}(\X) \to \X$.

We now confirm very concretely that $f^*({\sf q})$ is indeed a differential bundle.  Later, when we consider ``display'' tangent categories for which $\sf q$ is a ``display differential bundle'', we shall see a more conceptual proof of this result (remark \ref{remarkPullbackDiffBundles}).

\begin{lemma}\label{pullbackDiffBundle}
In any tangent category, when the pullback of ${\sf q}$ along $f$ exists and is preserved by $T^n$ (in the sense described above), $f^*({\sf q})$ is a differential bundle and the linear morphism $f^{*}_{\sf q}$ is a Cartesian 
morphism (in the fibrational sense) for the functor $P: {\sf DBun}(\X) \to \X$ sitting above $f$.
\end{lemma}

\begin{proof}
$f^{*}(q): f^{*}(E) \to X$ exists by assumption and, furthermore, pullback powers of this map exist and are preserved by $T^n$ by assumption.  The pullback of an additive bundle is 
always an additive bundle, thus, the only issue is the behaviour of the lift $f^{*}(\lambda)$.  Above we have already seen that $(f^{*}(\lambda),0)$ is an additive 
bundle morphism.  We must also show that $(f^{*}(\lambda),f^{*}(\zeta))$ is an additive bundle morphism.  The map $f^{*}(\zeta)$ is the unique map 
to the pullback in the lower back square of the following diagram.  Just to obtain the basic data of a bundle morphism, we must start by showing that 
the left square face commutes.
$$\xymatrix{ & T(f^{*}(E)) \ar[dd]|{\hole}_<<<<p \ar[rr]^{T(f^{*}_E)} & & T(E) \ar[dd]^{p} \\ 
           f^{*}(E) \ar[dd]_{f^{*}(q)}  \ar[ur]^{f^{*}(\lambda)} \ar[rr]^>>>>>>>>>{f^{*}_E} & & E \ar[dd]^<<<<<<q \ar[ru]_{\lambda} \\
           & f^{*}(E) \ar[dd]_<<<<<<<<{f^{*}(q)}|\hole \ar[rr]_<<<<<<{f^{*}_E}|>>>>>>>>>>>>>>>{~~} & & E \ar[dd]^q \\
           X \ar@{=}[dr] \ar@{..>}[ur]^{f^{*}(\zeta)} \ar[rr]_>>>>>>>{f} & & M \ar@{=}[dr] \ar[ur]_{\zeta} \\
           & X \ar[rr]_{f} & & M}$$
However this square is a map to the apex of a pullback.  Thus, the square commutes if we can show that $f^{*}\!(\lambda)pf^{*}_E = f^{*}\!(q)f^{*}(\zeta)f^{*}_E$ 
and $f^{*}(\lambda)pf^{*}(q) = f^{*}(q)f^{*}(\zeta)f^{*}(q)$.  Here are the calculations:
\begin{eqnarray*}
f^{*}(\lambda)pf^{*}_E & = & f^{*}(\lambda)T(f^{*}_E)p = f^{*}_E \lambda p \\
& = & f^{*}_E q \zeta = f^{*}(q)f \zeta = f^{*}(q)f^{*}(\zeta)f^{*}_E \\
f^{*}(\lambda)pf^{*}(q) & = & f^{*}(\lambda)T(f^{*}(q))p = f^{*}(q)0 p = f^{*}(q) \\ 
& = & f^{*}(q)f^{*}(\zeta)f^{*}(q).
\end{eqnarray*}
Next we have to show that $(f^{*}(\lambda),f^{*}(\zeta)): f^{*}({\sf q}) \to {\sf p}_{f^{*}(E)}$ is an additive bundle morphism.  However, the 
codomain of this morphism is the apex of a pullback of additive bundle morphisms:
$$\xymatrix{{\sf p}_{f^{*}(E)} \ar[d]_{(T(f^{*}(q)),f^{*}(q))} \ar[rr]^{(T(f^{*}_E),f^{*}_E)} & & {\sf p}_E \ar[d]^{(T(q),q)} \\
            {\sf p}_X \ar[rr]_{(T(f),f)} & & {\sf p}_M }$$
Thus, it suffices to show that the morphisms $(f^{*}(\lambda),f^{*}(\zeta))(T(f^{*}(q)),f^{*}(q))$ and $(f^{*}(\lambda),f^{*}(\zeta))(T(f^{*}_E),f^{*}_E)$ are additive bundle morphisms.  The first expression equals $(f^{*}(q),1_X)(0,1_X)$ -- the final bundle morphism 
followed by the zero bundle morphism -- which is certainly additive.  The second expression equals $(f^{*}_E,f)(\lambda,\zeta)$ which is also certainly an additive bundle morphism.  

Next we must show that $f^{*}(\lambda)\ell = f^{*}(\lambda)T(f^{*}(\lambda))$. As these maps have codomain $T^2(f^{*}(E))$, which is the apex of a pullback, 
they are equal if post-composing with $T^2(f^{*}_E)$ and $T^2(f^{*}(q))$ makes them equal.  Here are the calculations:
\begin{eqnarray*}
f^{*}(\lambda)\ell T^2(f^{*}_E) & = & f^{*}(\lambda) T(f^{*}_E)\ell = f^{*}_E\lambda \ell = f^{*}_E\lambda T(\lambda) \\
& = & f^{*}(\lambda) T(f^{*}_E)T(\lambda) = f^{*}(\lambda) T(f^{*}_E \lambda)\\
& = & f^{*}(\lambda) T( f^{*}(\lambda) T(f^{*}_E)) = f^{*}(\lambda) T(f^{*}(\lambda)) T^2(f^{*}_E) \\
f^{*}(\lambda)\ell T^2(f^{*}(q)) & = & f^{*}(\lambda) T(f^{*}(q)) \ell = f^{*}(q)0 \ell = 0 \ell T^2(f^{*}(q)) \\
& = &  f^{*}(\lambda)T(f^{*}(\lambda))T^2(f^{*}(q)).
\end{eqnarray*}

Finally we must prove the universality of the lift: this is given by the following 
$$\xymatrix{ & T(f^{*}(E)) \ar[dd]|{\hole}^<<<<<<{T(f^{*}(q))} \ar[rr]^{T(f^{*}_E)} & & T(E) \ar[dd]^{T(q)} \\ 
           f^*(E_2) \ar[dd]_{f^{*}(\pi_0 q)}  \ar@{..>}[ur]^{f^{*}(\mu)} \ar[rr] & & E_2 \ar[dd]^<<<<<<{\pi_0 q} \ar[ru]_{\mu} \\
           & T(X) \ar[rr]_<<<<<<{T(f)}|>>>>>>>>>>>>{~~} & & T(M) \\
           X \ar[ur]^{0} \ar[rr]_{f} & & M \ar[ur]_{0}}$$
in which the front, back, and left faces are pullbacks (all preserved by $T^n$) so that there is a unique $f^{*}(\mu)$ completing the cube making all the vertical squares 
pullbacks (preserved by $T^n$).  It is easily checked that $f^{*}(\mu)$ is the $\mu$ map for $f^{*}({\sf q})$.
\end{proof}

An example of an application of this lemma is in forming the ``Whitney sum'' of two differential bundles. Given two differential bundles ${\sf q}$, ${\sf q'}$ over the same base $M$ in a Cartesian tangent category, the Whitney sum may be formed by taking the pullback of ${\sf q} \times {\sf q}'$ along the diagonal:
\[
\bfig
	\square<500,350>[E \x_M E'`E \times E'`M`M \times M;``q \times q'`\Delta]
\efig
\] 

It is easy to see that in a Cartesian tangent category ${\sf q} \times {\sf q}'$ is a differential bundle.   However, we must know that pullbacks of differential bundles along the diagonal map exist and are preserved by $T^n$ in order to form the pullback bundle.  In general this is not guaranteed.  There are however more general situations in which the Whitney sum exists and we shall return to this later when we discuss display systems (see section \ref{Display}).


\subsection{Properties of differential bundles}

Many of the basic properties of the tangent bundle in a tangent category (which were described in \cite{sman3}, section 2.5) can be generalized to differential bundles.  This section collects these generalizations, as they are very useful when concretely working with differential bundles.  

To begin, note that for a differential bundle ${\sf q}$ on $E$, $T(E)$ has two addition operations: $T(\sigma)$ and $+$.  Our first observation concerns the conditions under which one can 
interchange these operations:

\begin{lemma}\label{lemmaInterchange}(Interchange of addition)
If ${\sf q}$ is a differential bundle, then for maps $v_1, v_2, v_3, v_4: X \to TE$:
\begin{enumerate}[(i)]
\item Whenever both sides are defined, we can interchange the additions $T(\sigma)$ and $+$:
$$\<\<v_1,v_2\>T(\sigma), \<v_3, v_4\>T(\sigma)\>+ = \<\<v_1, v_3\>+, \<v_2, v_4\>+\>T(\sigma)$$
\item When $v_1T(q\zeta) = v_1p0$ and $v_2T(q \zeta) = v_2p0$ (that is, when $v_1$ and $v_2$ share a common zero) and when both sides are defined, then:
	\[ \<v_1,v_2\>T(\sigma) = \< v_1,v_2\>+_T \] 
\end{enumerate}
\end{lemma}

\begin{proof}
\begin{enumerate}[{(i)}]
\item
The requirement that both sides of this equation be defined amounts to requiring that the following equations hold:
\[ v_1T(q) = v_2T(q), v_3T(q) = v_4T(q), v_1p = v_3p, \mbox{ and } v_2p = v_4p. \]

We have:
	\begin{eqnarray*}
	\lefteqn{\<\<v_1,v_2\>T(\sigma), \<v_3, v_4\>T(\sigma)\>+ } \\
	& = & \<\<v_1,v_2\>, \<v_3, v_4\>\>T_2(\sigma)+ \\
	& = & \<\<v_1,v_2\>, \<v_3, v_4\>\>+_{T_2}T(\sigma) \mbox{ (naturality of +)} \\
	& = & \<\<v_1, v_3\>+_T, \<v_2, v_4\>+_T\>T(\sigma) \mbox{ (see below).}
	\end{eqnarray*}
	For the last step, using the naturality of $+$, $+_{T_2}T(\p_0) = T_2(\p_0)+_T$ and  $+_{T_2}T(\p_1) = T_2(\p_1)+_T$, we have:
         \begin{eqnarray*}
	\lefteqn{\<\<v_1, v_2\>, \<v_3, v_4\>\>+_{T_2}} \\ & = & \<\<v_1, v_2\>, \<v_3, v_4\>\>\<T_2(\p_0)+_T, T_2(\p_1)+_T\> \\
	& = & \<\!\<v_1, v_2\>,\!\<v_3, v_4\>\!\> \<\!\<\p_0T(\p_1), \!\p_1T(\p_0)\>+_T,\!\<\p_0T(\p_1),\!\p_1T(\p_1)\>+_T\> \\
	& = & \<\<v_1, v_3\>+_T, \<v_2, v_4\>+_T\>.
	\end{eqnarray*}
\item
The requirement that both sides of this equation be defined amounts to requiring $v_1p=v_2p$ and $v_1T(q)=v_2T(q)$.  We then use an Eckmann-Hilton argument:
\begin{eqnarray*}
\<v_1,v_2\>T(\sigma) 
& = & \<\< v_2p0 v_1\>+,\<v_2,v_1p0\>+\>T(\sigma) \\
& = & \<\<v_2p0,v_2\>T(\sigma),\<v_1,v_1p0\>T(\sigma)\>+ ~~~\mbox{(interchange)} \\
& = & \<\<v_2T(q\zeta),v_2\>T(\sigma),\<v_1,v_1T(q\zeta)\>T(\sigma)\>+ \\
& = & \< v_2,v_1\>+
\end{eqnarray*}
\end{enumerate}
\end{proof}

Next, we develop some useful identities concerning the map $\mu$:

\begin{lemma}\label{lemmaV}
If ${\sf q}$ is a differential bundle then $\mu p = \pi_1$ and $\lambda = \<1,q\zeta\>\mu$.
\end{lemma}

\begin{proof}
For the first claim:
\begin{eqnarray*}
\mu p & = & \<\pi_0 \lambda,\pi_1 0\>T(\sigma)p \\
& = & \<\pi_0 \lambda p,\pi_1 0p\> \sigma \mbox{ (naturality of $p$)} \\
& = & \<\pi_0 q \zeta,\pi_1\>\sigma \mbox{ ($(\lambda,\zeta)$ a bundle morphism)} \\
& = & \pi_1 \mbox{ (unit of addition)}
\end{eqnarray*}
For the second claim: 
\[ \<1,q\zeta\>\mu = \<1,q\zeta\>\<\pi_0 \lambda,\pi_1 0\>T(\sigma) = \<\lambda,q\zeta0\>T(\sigma) = \<\lambda,q0T(\zeta)\>T(\sigma) = \lambda. \]
\end{proof}

The following result generalizes Lemma 2.12 in \cite{sman3}: 

\begin{lemma}\label{lemmaVerticalLift}
In the presence of the other axioms for a differential bundle ${\sf q}$, the universality of the lift may be equivalently expressed by demanding either of the following:
\begin{enumerate}[(i)]
\item
$$E_2 \to^{\mu} T(E) \Two^{T(q)}_{pq0} T(M)$$
is an equalizer;
\item for any map $f: X \to TE$ such that $fT(q) = fpq0$, there is a unique map $\{f\}: X \to E$ such that
	 \[ f = \<\{f\}\lambda, f p0 \>T(\sigma). \]
\end{enumerate}
\end{lemma}

\begin{proof}
\begin{enumerate}[{\em (i)}]
\item Assuming the pullback,  first note that the $\mu T(q) = \pi_1 q 0 = \mu p q 0$ so $\mu$ equalizes the two maps.  
Now given $g$ with $gT(q) = gpq0$ then this gives a unique map $g|_\mu$ which mediates the pullback.  

Conversely, assuming the equalizer then $\mu T(q) = \mu p q 0 = \pi_1 q 0$ and so the square commutes. If $g T(q) = g' 0$ then 
$g T(q) = g'0 = g'0p0 = gT(q)p0 = gpq0$ so there is a unique map to the equalizer which makes the square a pullback.

\item Since $\mu$ is the equalizer, we have a unique map $f|_{\mu}: X \to E_2$ such that $f|_{\mu}\mu = f$.  We set $\{f\} := f|_{\mu}\p_0$ and have:
\begin{eqnarray*}
f & = & f|_{\mu} \mu \\
& = & f|_{\mu} \<\p_0 \lambda, \p_1 0\> T(\sigma) \\
& = & \<f|_{\mu}\p_0 \lambda, f|_{\mu}\p_1 0\> T(\sigma) \\
& = & \<\{ f \} \lambda, f|_{\mu} \mu p 0\> T(\sigma) \mbox{ (by lemma \ref{lemmaV})} \\
& = & \< \{ f \} \lambda, f p 0\> T(\sigma) 
\end{eqnarray*}
The uniqueness of $\{f\}$ follows from the uniqueness of $f|_{\mu}$ as
$$f = \< \{ f \}\lambda,fp0\> T(\sigma) = \< \{f\},f p\>  \<\p_0 \lambda, \p_1 0\> T(\sigma) = \<\{f\},fp\>{\mu}$$
so that $\< \{f\},f p \> = f|_{\mu}$.
\end{enumerate}
\end{proof}
Note that for a trivial differential bundle $M \to^{1} M$ and a map $f: X \to TM$, $\{f\}$ is just $fp$. \\

Because $T$ (and all powers of $T$) preserves the pullback expressing the universality of lift, it follows that: 

\begin{corollary} 
For a differential bundle ${\sf q}$ in a tangent category, $T$ and all powers of $T$ preserve the equalizer
$$E_2 \to^\mu T(E) \two^{T(q)}_{pq0} T(M).$$
\end{corollary}

We shall have occasion to use the bracketing operation introduced in Lemma \ref{lemmaVerticalLift} and so 
it is useful to establish some of its key properties: 

\begin{lemma} \label{lemmaLift}
For $f, g: X \to TE$ which equalize $T(q)$ and $pq0$ and have $fpq = gpq$:
\begin{enumerate}[(i)]
	\item for any $k: Z \to X$, $k\{ f\} = \{ kf\}$;
        \item Suppose $(h,k): {\sf q} \to {\sf q}'$ is a linear bundle morphism, and $x: X \to T(E)$ equalizes $T(q)$ and $pq0$.  
Then $\{x\}h = \{xT(h)\}$;
	\item $\{f\}q = fT(q)p$ when the left hand side is defined;
	\item $\{0\} = q\zeta$;
	\item $\<\{f\},\{g\}\>\sigma = \{\<f,g\>T(\sigma)\} $ when either side is defined;
	\item $\<\{f\},\{g\}\>\sigma = \{ \<f,g\>+\}$ when  both sides are defined;
	\item $\{ \mu \} = \pi_0$ and $\{\lambda\} = 1$. 
\end{enumerate}
\end{lemma}

\begin{proof}
\begin{enumerate}[{\em (i)}]
\item $k\{f\} = kf|_v\pi_0 = (kf)|_v \pi_0 = \{kf\}. $
\item We must first check that $\{xT(h)\}$ is valid term; that is that $xT(h)$ equalizes $T(q')$ and $pq'0$:
\begin{eqnarray*}
xT(h)T(q') & = & xT(hq') = xT(qk) = xT(q)T(k) \\
& = & xT(q)p0T(k) \mbox{ (by assumption)} \\
& = & xT(q)pk0 = xT(q)T(k)p0 = xT(qk)p0 \\
& = & xT(hq')p0 = xT(h)T(q')p0 = xT(h)pq'0
\end{eqnarray*}
To show that $\{x\}h = \{xT(h)\}$, we need to show that $\{x\}h$ satisfies the same universal property as $\{xT(h)\}$:
\begin{eqnarray*}
&   & \<\{x\}h\lambda',xT(h)p0\>T(\sigma) \\
& = & \<\{x\}\lambda T(h),fp0T(h)\>T(\sigma) \\ 
~&~&~~~~~~~~~~\mbox{ ($g$ is a linear bundle morphism and naturality of $p0$)} \\
& = & \<\{x\}\lambda,xp0\>\<T(\pi_0)T(h),T(\pi_1)T(h)\>T(\sigma) \\
& = & \<\{x\}\lambda,xp0\>T(\sigma)T(h) \mbox{ ($T(h)$ is an additive bundle morphism)} \\
& = & xT(h)
\end{eqnarray*} 
\item We have:
\begin{eqnarray*}
fT(q)p & = & \<\{f\}\lambda, fp0\>T(\sigma)T(q)p \\
& = & \<\{f\}\lambda, fp0\>T(\pi_0)T(q)p \mbox{ ($\sigma$ is a bundle morphism)} \\
& = & \{f\}\lambda T(q) p \\
& = & \{f\}q0p \mbox{ ($(\lambda,0)$ is additive)} \\
& = & \{f\}q \mbox{ (since $0p = 1$)}.
\end{eqnarray*}

\item We need to show $q\zeta$ has the same universal property as $\{0\}$:
\begin{eqnarray*}
 \<q \zeta \lambda,0p0\>T(\sigma) 
& = & \<q0T(\zeta),0\>T(\sigma) \mbox{ (additivity of $\lambda$)} \\
& = & \<0T(q)T(\zeta),0\>T(\sigma) \mbox{ (naturality of $0$)} \\
& = & 0\<T(q)T(\zeta),1\>T(\sigma) \\
& = & 0 \mbox{ (addition of a zero term)}
\end{eqnarray*}

\item Suppose that $\<\{f\},\{g\}\>\sigma$ is well-defined.  Since $\{f\}$ and $\{g\}$ are well-defined, $fT(q) = fT(q)p0$ and $gT(q) = gT(q)p0$.  Since $\<\{f\},\{g\}\>\sigma$ is well-defined, $\{f\}q = \{g\}q$ so by (ii) $fT(q)p = gT(q)p$.  Thus
	\[ fT(q) = fT(q)p0 = gT(q)p0 = gT(q) \]
so $\<f,g\>T(\sigma)$ is well-defined, and
	\[ \< f,g\>T(\sigma)T(q) = fT(q) = fT(q)p0 = \< f,g\>T(\sigma)T(q)p0 \]
so that $\{ \<f,g\>T(\sigma)\}$ is defined.

Conversely, suppose that $\{\< f,g\>T(\sigma)\}$ is defined.   Then $fT(q)=gT(q)$ and $\< f,g\>T(\sigma)T(q) = \<f,g\>T(\sigma)T(q)p0$.  Thus
	\[ fT(q) = \<f,g\>T(\sigma)T(q) = \<f,g\>T(\sigma)T(q)p0 = fT(q)p0 \]
so $\{f\}$ and similarly $\{g\}$ is defined.  Finally, by (ii)
	\[ \{f\}q = fT(q)p = gT(q)p = \{g\}p \]
so $\<\{f\},\{g\}\>\sigma$ is defined.  

Thus one side is defined if and only if the other side is.  

It remains  to show $\<\{f\},\{g\}\>+$ has the same universal property as $\{\<f,g\>T(\sigma)\}$:
\begin{eqnarray*}
\lefteqn{\<\<\{f\},\{g\}\>+\lambda, \<f,g\>T(\sigma)p0\>T(\sigma)} \\
& = & \<\<\{f\}\lambda,\{g\}\lambda\>T(\sigma), \<fp0, gp0\>T(\sigma)\>T(\sigma) \\
&~&~~~~~~~~~~\mbox{ ($\lambda$ is a morphism of bundles)} \\
& = & \<\<\{f\}\lambda,fp0\>T(\sigma), \<\{g\}\lambda, gp0\>T(\sigma)\>T(\sigma) \\
&~&~~~~~~~~~~\mbox{ (associativity and commutativity)} \\
& = & \<f,g\>T(\sigma).
\end{eqnarray*}
\item As in the previous result, we just need to show that $\<\{f\},\{g\}\>\sigma$ has the same universal property as $\{\<f,g\>+\}$:
\begin{eqnarray*}
\lefteqn{\<\<\{f\},\{g\}\>\sigma \lambda, \<f,g\>+p0\>T(\sigma)} \\
& = & \<\<\{f\}\lambda,\{g\}\lambda\>+, fp0\>T(\sigma)  ~~\mbox{($\lambda$ is a morphism of bundles)} \\
&= & \<\<\{f\}\lambda,\{g\}\lambda\>+, \< fp0, gp0\> + \>T(\sigma)\\
&~&~~~~~~~~~~\mbox{(as $fp=gp$ and adding of zero)}\\
& = & \<\< \{ f \}\lambda,fp0\> T(\sigma),\< \{ g \}\lambda,gp0\> T(\sigma) \>+\\
&~&~~~~~~~~~~\mbox{(interchange from lemma \ref{lemmaInterchange})}\\
& = & \<f, g\> +.
\end{eqnarray*}
 \item The first is immediate from the universality of $\mu$ but this means by lemma \ref{lemmaV} that
	\[ \{\lambda\} = \{\<1,q\zeta\>\mu\} = \<1,q\zeta\>\{\mu\} = \<1,q\zeta\>\pi_0 = 1. \]
\end{enumerate}
\end{proof}

The above generalizes most of lemma 2.14 of \cite{sman3} to differential bundles, with in particular {\em (vii)} of lemma 2.14 generalizing to {\em (ii)} in the above.  The only result not generalized is 2.14{\em (vi)}.  This is generalized by the following result, which relates the bracketing operation for the differential bundle $T({\sf q})$ (see Lemma \ref{lemmaTDiffBundle}) to the bracketing operation in the differential bundle $\sf q$:
\begin{lemma}\label{lemmaBracketT0}
If ${\sf q}$ is a differential bundle and $f: X \to T(E)$ has $f T(q) = fpq0$ then $T(\{f\}) = \{T(f)c\}$.
\end{lemma}
\begin{proof}
We first need to show $\{T(f)c\}$ is well-defined; that is, we need to show that 
	\[ T(f)cT^2(q) = T(f)cT^2(q)p0. \]
Indeed,
\begin{eqnarray*}
T(f)cT^2(q) 
& = & T(f)T^2(q)c \mbox{ (naturality of $c$)} \\
& = & T(fT(q))c \\
& = & T(fT(q)p0)c \mbox{ (assumption on $f$)} \\
& = & T(f)T^2(q)T(p)T(0)c \\
& = & T(f)T^2(q)cp0 \mbox{ (coherences on $c$)} \\
& = & T(f)cT^2(q)p0 \mbox{ (naturality of $c$)}
\end{eqnarray*}
We now check $T(\{f\})$ has the same universal property as $\{T(f)c\}$:
\begin{eqnarray*}
\lefteqn{\<T(\{f\})T(\lambda)c, T(f)cp0\>T^2(\sigma) }\\
& = & \<T(\{f\}\lambda)c,T(f)T(p)T(0)c \> T^2(\sigma) \mbox{ (coherence for $c$)} \\
& = & \<T(\{f\}\lambda),T(fp0) \> T^2(\sigma)c \mbox{ (naturality of $c$)} \\
& = & T(\<\{f\}\lambda,fp0\>T(\sigma)c \\
& = & T(f)c
\end{eqnarray*}
as required. 
\end{proof}

Lemma 2.14{\em (vi)\/} of \cite{sman3} states $T(\{ f\}) = \{T(f)cT(c)\}c$ when the left side is defined.   The fact that $c$ is part of a linear morphism with \ref{lemmaLift}{\em (ii)} gives $\{T(f)cT(c)\}c = \{T(f)cT(c)T(c)\} = \{T(f)c \}$ which allows one to inter-derive the two results.

The next result generalizes lemma 2.13 of \cite{sman3}.

\begin{lemma} \label{vertical-lift-as-equaliser}
For a differential bundle, ${\sf q}$, the following is a joint-equalizer diagram:
	\[\xymatrix{ & & & T(M) \\ E \ar[r]^{\lambda} & T(E)   \ar@<1ex>[rru]^{T(q)} \ar@<-1ex>[rru]_{pq0} \ar@<1ex>[rrd]^{p} \ar@<-1ex>[rrd]_{pq\zeta} \\ & & & E} \]
and thus, in particular, $\lambda$ is monic.
\end{lemma}

\begin{proof}
First, $\lambda$ equalizes the above maps since $(\lambda,0)$ and $(\lambda,\zeta)$ are bundle morphisms:
	\[ \lambda T(q)p0 = q0p0 = q0 = \lambda T(q) \]
and
	\[ \lambda pq \zeta = q\zeta q \zeta = q\zeta = \lambda p. \]

If $f: X \to TE$ equalizes the above two maps, then in particular $fT(q) = fT(q)p0$, so by lemma \ref{lemmaVerticalLift}(i), there exists a unique map $\{f\}: X \to E$ with the property that 
	\[ f = \<\{f\}\lambda,fp0\>T(\sigma) \]
We claim that $\{f\}$ is the required unique map to show universality of $\lambda$.  Starting with the above equation, we get
\begin{eqnarray*}
f & = & \<\{f\}\lambda,fp0\>T(\sigma) \\
& = & \<\{f\}\lambda,fpq\zeta 0\>T(\sigma) \mbox{ (by assumption on f)} \\
& = & \<\{f\}\lambda,fT(q)p0T(\zeta)\>T(\sigma) \mbox{ (naturality of $p$ and $0$)} \\
& = & \<\{f\}\lambda,\{f\}q0T(\zeta)\>T(\sigma) \mbox{ (by Lemma \ref{lemmaLift} (iii))} \\
& = & \<\{f\}\lambda,\{f\}\lambda T(q)T(\zeta)\>T(\sigma) \mbox{ ($(\lambda,0)$ is a bundle morphism)} \\
& = & \{f\}\lambda\<1,T(q)T(\zeta)\>T(\sigma) \\
& = & \{f\}\lambda \mbox{ (addition of zero term)} 
\end{eqnarray*}
as required.
\end{proof}

The following is also a useful observation:

\begin{lemma} If ${\sf q}$ is a differential bundle then
$$M \to^\zeta E \Two^{0}_{\lambda} T(E)$$
is an equalizer.
\end{lemma}

\begin{proof}
$\zeta$ is a section so monic, thus it suffices to show that if $x 0 = x \lambda$ that $x$ factors through $\zeta$.  
To this end we have:
$$x q \zeta = x \lambda p = x 0 p = x.$$
\end{proof}

\subsection{Properties of linear bundle morphisms}

Morphisms of differential bundles are not required to preserve the additive structure; they simply commute with the projections of the bundles.  Bundle morphisms which preserve the lift, that is, the linear bundle 
morphisms, are of fundamental importance.  They generalize the linear maps of Cartesian differential categories (we shall see this later in proposition \ref{linearDiffOver1}), and so also generalize linear maps in the ordinary 
sense.  

In this subsection we prove some basic properties of these linear bundle morphisms. 

\begin{proposition} \label{linearisadditive} 
Linear morphisms of differential bundles are additive.
\end{proposition}

\begin{proof}
If $(f,g): (q,\sigma,\zeta,\lambda) \to (q',\sigma',\zeta',\lambda')$ is a linear morphism of bundles (so that $\lambda T(f) = f\lambda'$) then 
to say it is additive is to require that $\sigma f = \<\pi_0f,\pi_1f\>\sigma'$ and $\zeta f = g \zeta'$.
To show the first equality, we post-compose by $\lambda'$ and use the fact that $\lambda'$ is monic 
(see Lemma \ref{vertical-lift-as-equaliser}):
\begin{eqnarray*}
\sigma f \lambda' 
& = & \sigma \lambda T(f) \mbox{ (linearity of $f$)} \\
& = & \<\pi_0 \lambda,\pi_1 \lambda\>+T(f) \mbox{ ($\lambda$ is additive)} \\
& = & \<\pi_0 \lambda,\pi_1 \lambda\>\<\pi_0 T(f),\pi_1 T(f)\>+ \mbox{ ($T(f)$ is additive)} \\
& = & \<\pi_0 \lambda T(f),\pi_1 \lambda T(f)\>+ \\
& = & \<\pi_0 f \lambda',\pi_1 f \lambda'\>+ \mbox{ (linearity of $f$)} \\
& = & \<\pi_0 f,\pi_1 f\>\<\pi_0 \lambda',\pi_1\lambda'\>+ \\
& = & \<\pi_0 f,\pi_1 f\>\sigma \lambda' \mbox{ ($\lambda'$ is additive)} 
\end{eqnarray*}
For the preservation of zero, we use the fact that $q$ is epic:
\begin{eqnarray*}
q \zeta f  
& = & \lambda p f = \lambda T(f) p = f \lambda' p \\
& = &  f q' \zeta' = q g \zeta'.
\end{eqnarray*}
\end{proof}

This generalizes the fact that linear maps in a Cartesian differential category are always additive (lemma 2.2.2{\em (i)} in \cite{cartDiff}).

The following shows that how linearity of a bundle morphism is related to preservation of $\mu$:

\begin{lemma}\label{mulambda} Suppose $(f,g): {\sf q} \to {\sf q'}$ is a bundle morphism.  If $(f,g)$ is linear, then the identity $\mu T(f) = \< \pi_0f,\pi_1 f\>\mu'$ holds.
Conversely, if this identity holds and the bundle morphism preserves zero, that is $\zeta f = g \zeta'$, then the bundle morphism is linear.  
\end{lemma}
\begin{proof}
If $(f,g)$ is a linear bundle morphism it is additive so that:
\begin{eqnarray*}
\mu T(f) & = & \<\pi_0 \lambda, \pi_1 0 \> T(\sigma) T(f) \\
& = & \<\pi_0 \lambda, \pi_1 0 \> T(\sigma f) \\
& = & \<\pi_0 \lambda, \pi_1 0 \> T((f \times f)\sigma') \\
& = & \<\pi_0 \lambda, \pi_1 0 \> (Tf \times Tf)\sigma') \mbox{ ($T$ preserves the pullback for $E_2$)} \\
& = & \<\pi_0 \lambda T(f), \pi_1 0 T(f) \> T(\sigma') \\
& = & \<\pi_0 f \lambda', \pi_1 f0 \> T(\sigma') \\
& = & (f \times f)\<\pi_0 \lambda', \pi_1 0\>T(\sigma') \\
& = & (f \times f)\mu'
\end{eqnarray*}
Conversely, if $(f,g)$ satisfies the above identity and preserves zero, then pre-composing with $\<1,q\zeta\>$ gives
\begin{eqnarray*}
 \lambda T(f) & = & \<1,q\zeta\>\mu T(f)  =  \<1,q\zeta\>(f \times f)\mu' \\
& = & \<f,q\zeta f\>\mu'  =  \<f,qg\zeta'\>\mu' \\
& = & \<f,fq'\zeta'\>\mu'  =  f\<1,q'\zeta'\>\mu' \\
& = & f\lambda'
\end{eqnarray*}
so that $(f,g)$ is a linear bundle morphism.
\end{proof}

Another useful result is that the inverse of a linear bundle morphism is automatically linear:

\begin{lemma}\label{lemmaLinearInverses}
Suppose $(f,g): \sf{q} \to \sf{q'}$ is a linear bundle morphism.
\begin{enumerate}[(i)]
	\item If $(f,g)$ is a retract and $(h,k): \sf{q'} \to \sf{q''}$ is a bundle morphism such that $(fh,gk)$ is linear, then $(h,k)$ is linear.
	\item If $(f,g)$ is a bundle isomorphism then $(f^{-1},g^{-1})$ is linear.
\end{enumerate}
\end{lemma}
\begin{proof}
\begin{enumerate}[{\em (i)}]
\item Let $(f',g')$ be a section of $(f,g)$.  Then since $(f,g)$ and $(fh,gk)$ are linear we have
	\[ \lambda' T(h) = f'f \lambda''T(h) = f'\lambda T(f)T(h) = f' \lambda T(fh) = f'f h \lambda'' = h \lambda'' \]
so that $(h,k)$ is linear, as required.
\item This follows immediately from (i) since the identity bundle morphism is linear.
\end{enumerate}
\end{proof}


\section{Differential structure} \label{diff-structure}


In a Cartesian tangent category one can consider differential bundles over the final object.  As we shall see, these special differential bundles have an 
alternate description as a {\em differential object\/}, first defined in \cite{sman3}.   When the tangent structure behaves well with respect to slicing, as is the case in synthetic 
differential geometry (see Example \ref{basicdiffbundles} (4)), it is possible to use the structure of a differential object in the slice $\X/M$ as the {\em definition\/} 
of a differential bundle over $M$.   However, as discussed in the introduction, securing the necessary behaviour with respect to slicing requires some delicacy, and will be returned to in later sections of this paper.

Differential objects were developed in order to relate tangent categories to Cartesian differential categories.  In \cite{sman3} it was shown that the differential 
objects of any Cartesian tangent category always form a Cartesian differential category.   In section \ref{diff-fibrations}, we  will need to know when 
a Cartesian tangent category {\em is\/} a Cartesian differential category.  Sorting this out is the main and rather technical objective of this section.  The problem is this: clearly a Cartesian tangent category which is a Cartesian differential category must have associated to every object a ``canonical'' differential structure.  However, 
objects may possess multiple differential structures and an arbitrary choice of differential structure for each object will not necessarily be compatible with extant 
tangent structure.  There must, therefore, be a manner of choosing from the possible differential structures which ``fits'' with the extant tangent structure.  We call such a choice a {\em coherent\/} choice of differential structure.  Furthermore, we describe precise requirements for a choice to be coherent.  We then show that a Cartesian 
tangent category with such a coherent choice of differential structure is a Cartesian differential category.  This rather technical result 
underpins the transition from tangent categories to Cartesian differential categories which is needed to establish one of the main results of this paper, Theorem \ref{display-is-diff-fib}, which shows that there is a tangent fibration of differential bundles, and in this fibration, every fibre is a Cartesian differential category.

\subsection{Differential objects}


In \cite{sman3}, an important structure an object of a tangent category can possess, called  {\em differential structure\/}, was described.  Objects with differential structure are called {\em differential objects\/} and they play an analogous role to vector spaces in the category of smooth manifolds.  Of course, since a general tangent category has no analogue of the real numbers there is no notion of ``vector space'' as such.  Instead differential structure is defined based on another special property enjoyed by vector spaces: namely, that of having a trivial tangent bundle in the sense that $T(V) \cong V \times V$.  Since the tangent bundle already has one projection map, 
$p: T(A) \to A$, in order to define a differential object, it suffices to demand the existence of a further projection $\hat{p}: T(A) \to A$ satisfying certain properties.  

\begin{definition}
For an object $A$ in a Cartesian tangent category, \textbf{differential structure on $A$} consists of a commutative monoid structure 
$\sigma: A \times A \to A$, $\zeta: 1 \to A$ on $A$, (making $(!_A,\zeta,\sigma)$ an additive bundle over $1$), together with a map $\hat{p}: TA \to A$ such that 
\begin{itemize}
\item $A \from^{\hat{p}} TA \to^{p} A$ is a product diagram;
\item $(\hat{p},!_{T(A)}) : (TA,T(\zeta),T(\sigma)) \to (A,\zeta,\sigma)$ is an additive bundle morphism, that is the following diagrams 
commute: 
$$\xymatrix{T(1) \ar[d]_{T(\zeta)}  \ar[rr]^{!_{T(1)}} & & 1 \ar[d]^{\zeta} \\
            T(A) \ar[rr]_{\hat{p}} & & A} ~~~~~
\xymatrix{T(A \x A) \ar[d]^{T(\sigma)} \ar[rr]^{ \< T(\pi_0 \hat{p},T(\pi_1)\hat{p} \> } & & A \x A \ar[d]^{\sigma} \\
          T(A) \ar[rr]_{\hat{p}} & & A}$$
\item $(!_A,\hat{p}): (TA,0,+) \to (A,\zeta,\sigma)$ is an additive bundle morphism, that is the following diagrams commute:
$$\xymatrix{A \ar[d]_{0_A} \ar[rr]^{!_A} && 1 \ar[d]^{\zeta} \\ T(A) \ar[rr]_{\hat{p}} & & A}
 ~~~~~\xymatrix{T_2(A) \ar[d]_{+} \ar[rr]^{\< \pi_0\hat{p},\pi_1\hat{p}\>} & & A \x A \ar[d]^{\sigma} \\ T(A) \ar[rr]_{\hat{p}} & & A}$$

\item $\hat{p}$ is coherent with respect to the vertical lift in the sense that the following commutes:
$$\xymatrix{T(A) \ar[d]_{\hat{p}} \ar[rr]^\ell & & T^2(A) \ar[d]^{T(\hat{p})} \\ A & & T(A) \ar[ll]^{\hat{p}}}$$
\end{itemize}
An object together with a specified differential structure is called a \textbf{differential object}.
\end{definition}

The requirement that the map $T(\zeta)$ is preserved is actually implied by the fact that the map $0_A$ is preserved:

\begin{lemma}\label{linearityOfZeta}
If $(A,\sigma, \zeta)$ is a commutative monoid in a Cartesian differential category and $\hat{p}: T(A) \to A$ is any map 
then $T(\zeta)\hat{p} = !_{T(1)} \zeta$ if $0_A\hat{p} = !_A\zeta$.
\end{lemma}
\begin{proof}
Consider the calculation:
\begin{eqnarray*}
!_{T(1)}\zeta & = & p_1 \zeta !_A \zeta \mbox{ \ (uniqueness of terminal objects)} \\
& = & p_1 \zeta 0_A \hat{p} \mbox{ \ (as $0_A\hat{p} = !_A\zeta$)} \\
& = & p_1 0_1 T(\zeta) \hat{p} \mbox{ \ (naturality of 0)} \\
& = & T(\zeta) \hat{p} \mbox{ \ ($T(1)$ is terminal)} .
\end{eqnarray*}
\end{proof}

Thus, the form of this definition has some redundancy, which was exploited in the original definition given in \cite{sman3}.  Here we have chosen a more natural presentation which displays the additive 
bundle morphisms involved.  It should also be noted that the definition in \cite{sman3} omitted the important axiom $\ell T(\hat{p})\hat{p} = \hat{p}$. 

\begin{example}{\em ~
\begin{enumerate}[{\em (i)}]
	\item In the category of finite-dimensional smooth manifolds, the differential objects are vector spaces, $\mathbb{R}^n$, as their tangent bundle is $T(\mathbb{R}^n) \simeq \mathbb{R}^n \x \mathbb{R}^n$.
	\item Similarly, in the category of convenient manifolds, the convenient vector spaces are differential objects.
	\item In the category of affine schemes, that is ($\mbox{cRing}^{\op}$), polynomial rings $\mathbb{Z}[x_1,x_2 \ldots x_n]$ are differential objects.
	\item In a model of SDG the ring of line type is always a differential object, as the requirement that $R^D \simeq R \x R$ is part of its definition.  More generally, in models of SDG differential objects are exactly 
	         Euclidean $R$-modules  (see section \ref{sectionRepDiffObjects}).
	\item In a Cartesian differential category (see section \ref{secCartDiffCats}) every object is canonically a differential object.
\end{enumerate}}
\end{example}

An alternative viewpoint to take on the definition of differential objects is as follows.  Since $A$ is a commutative monoid and $T$ preserves products, $T(A)$ is also a commutative monoid.  The first two requirements of a differential object then ask that $T(A)$ be a product in the category of commutative monoids in $\X$, $\mbox{cMon}(\X)$.  However, $\mbox{cMon}(\X)$ is an additive category and so in fact the first two requirements are equivalent to asking that $T(A)$ be a biproduct in $\mbox{cMon}(\X)$.  Thus, by general results about biproducts in an additive category (see, for example, \cite{macLane}, VIII.2), this part of the definition could equivalently be given by asking that $T(A)$ be a coproduct in $\mbox{cMon}(\X)$ (involving, in particular, a map $\lambda$ from $A$ to $T(A)$), or by asking that there be a $\hat{p}$ and a $\lambda$ satisfying certain equations.  These observations will be useful in the proof of:

\begin{proposition} \label{diffObjeqdiffBun}
In a Cartesian tangent category, to give a differential object is precisely to give a differential bundle over the final object.  
\end{proposition}


\begin{proof}
Given a differential object $(A,\hat{p},\sigma,\zeta)$, we use the universal property of the product $T(A)$ to set $\lambda = \<1,!\zeta\>$.  (Note that the other injection, $\<!\zeta,1\>$, is $0: A \to T(A)$ since $0\hat{p} = !\zeta$.)  As mentioned above, by general results about biproducts in an additive category, the universality of $\lambda$ and one of the additivity requirements for $\lambda$ follows automatically.  

For the other additivity requirements,
\begin{eqnarray*}
\<\pi_0\lambda,\pi_1\lambda\>+\hat{p} 
& = & \<\pi_0\lambda,\pi_1\lambda\>\<\pi_0\hat{p},\pi_1\hat{p}\>\sigma \mbox{ (coherence for $\hat{p}$)} \\
& = & \<\pi_0\lambda \hat{p},\pi_1\lambda \hat{p} \> \sigma \\
& = & \<\pi_0,\pi_1\>\sigma \mbox{ (by definition of $\lambda$)} \\
& = & \sigma \\
& = & \sigma\<1,!\zeta\>\hat{p}
\end{eqnarray*}
and
\[ \zeta \lambda \hat{p} = \zeta = \zeta ! \zeta = \zeta 0 \hat{p}. \]
The final coherence, $\lambda \ell_A = \lambda T(\lambda)$, asks for the equality of two maps into $T^2(A)$.  Since $T(A)$ is a product and $T$ preserves products, 
$T^2(A)$ is also a product with projections 
	\[ \<T(\hat{p})\hat{p}, T(\hat{p})p, T(p)\hat{p}, T(p)p\>. \]
Thus, it suffices to check the two sides of the equality are equal when post-composed by each of these four maps.  For the first, by the last coherence for $\hat{p}$,
	\[ \lambda \ell_A T(\hat{p}) \hat{p} = \lambda \hat{p} = 1 \]
while
	\[ \lambda T(\lambda)T(\hat{p})\hat{p} = \lambda T(\lambda \hat{p}) \hat{p} = \lambda \hat{p} = 1 \]
For the second,
	\[ \lambda \ell_A T(\hat{p})p = \lambda \ell_A p \hat{p} = \lambda p0 \hat{p} = ! \zeta ! \zeta = ! \zeta \]
while
	\[  \lambda T(\lambda)T(\hat{p})p = \lambda p = ! \zeta \]
Checking for post-composing with the other two maps is straightforward.

Conversely, if we have a differential bundle over $1$ $(A,\lambda,\sigma,\zeta)$, we set $\hat{p} := \{1\}$.  The universality of $\hat{p}$ and the first additive requirement follows from general biproduct results.  

For the other two additive requirements for a differential object, we have
	\[ 0\hat{p} = 0\{1\} = \{0\} = !\zeta \]
by lemma \ref{lemmaLift}(iii), and
	\[ +\hat{p} = +\{1\} = \{+\} = \{\<\pi_0,\pi_1\>+\} = \<\{\pi_0\},\{\pi_1\}\>\sigma = \<\pi_0\hat{p},\pi_1\hat{p}\>\sigma \]
by lemma \ref{lemmaLift}(vi).  

For the coherence with vertical lift we will use the fact (which, again, follows from general results about biproducts), that $\<\hat{p},p\>\mu = 1$:
\begin{eqnarray*}
\ell T(\hat{p})\hat{p} & = & \<\hat{p},p\>\mu \ell T(\hat{p})\hat{p} \\
& = & \<\hat{p},p\>\<\pi_0 \lambda, \pi_1 0\>T(\sigma) \ell T(\hat{p})\hat{p} \\
& = & \<\hat{p} \lambda, p 0\> \ell T^2(\sigma) T(\hat{p})\hat{p} \\
& = & \<\hat{p} \lambda, p 0\> \ell T(T(\sigma)\hat{p})\hat{p} \\
& = & \<\hat{p} \lambda, p 0\> \ell T(\<T(\pi_0) \hat{p}, T(\pi_1) \hat{p}\>\sigma)\hat{p}  \\
& = & \<\hat{p} \lambda, p 0\> \ell \<T^2(\pi_0)T(\hat{p}, T^2(\pi_1)T(\hat{p})\>T(\sigma)\hat{p} \\
& = & \<\hat{p} \lambda, p 0\> \<T(\pi_0) \ell T(\hat{p}, T(\pi_1)\ell T(\hat{p})\>\<T(\pi_0) \hat{p}, T(\pi_1) \hat{p}\>\sigma \\
& = & \<\hat{p} \lambda, p 0\> \<T(\pi_0) \ell T(\hat{p}, T(\pi_1)\ell T(\hat{p})\>\<T(\pi_0) \hat{p}, T(\pi_1) \hat{p}\>\sigma \\
& = & \<\hat{p}\lambda \ell T(\hat{p})\hat{p}, p0 \ell T(\hat{p})\hat{p}\>\sigma \\
& = & \<\hat{p}\lambda T(\lambda)T(\hat{p})\hat{p}, p0 0 T(\hat{p})\hat{p}\>\sigma  \\
& = & \<\hat{p}\lambda \hat{p}, p 0\hat{p} 0 \hat{p} \> \sigma \\
& = & \<\hat{p}, p! \zeta ! \zeta \>\sigma \\
& = & \hat{p}
\end{eqnarray*}

Thus differential bundles over $1$ and differential objects are in bijective correspondence.  
\end{proof}

Differential objects abound as one can construct a differential object from any differential bundle ${\sf q}$ (including the tangent bundle) by pulling back along a point of the base:

\begin{corollary}\label{corTangentSpaces}
In any Cartesian tangent category, if $q: E \to M$ is a differential bundle and  $a: 1 \to M$ is any point such that the pullback $E_a$ of $a$ along $q$ itself is preserved by $T^n$, then $E_a$ is a differential object.
\end{corollary}

\begin{proof}
By \ref{pullbackDiffBundle}, $E_a$ is a differential bundle over $1$, and so by the previous result is a differential object.  
\end{proof}

Note that this observation allows an alternate characterization of differential objects and it subsumes a key result of \cite{sman3} (theorem 4.15).


\subsection{Additional properties of differential objects}

One may wonder whether any coherence with the canonical flip $c$ should be required of differential objects (or differential bundles over $1$). In fact, some basic coherence is automatic:

\begin{proposition}\label{propCCondition}
For any differential bundle over $1$, $(A,\sigma,\zeta,\lambda)$, 
$$cT(\hat{p})\hat{p} = T(\hat{p})\hat{p}.$$
\end{proposition}

\begin{proof}
The proof is by calculation. It repeatedly uses the identity on $T(A)$ that we established in proposition \ref{diffObjeqdiffBun}:
$$1_{T(A)} = \< \hat{p},p\> \mu = \< \hat{p}\lambda,p0\>m^T_\x T(\sigma),$$  
where $m^T_\x = \<T(\pi_0,T(\pi_1)\>^{-1}$.

We start by proving an auxiliary fact:
$$T(\hat{p}\lambda)c T(\hat{p})\hat{p} = T(\hat{p})\hat{p}$$
in which it is useful to observe the following commutation:
$$\xymatrix{T_2(T(A)) \ar[d]_{+} \ar[rr]^{T_2(\hat{p})} & & T_2(A) \ar[d]^{+} \ar[rr]^{\< \pi_0\hat{p},\pi_1\hat{p} \>} & & A \x A \ar[d]^\sigma \\            
            T^2(A) \ar[rr]_{T(\hat{p})} & & T(A) \ar[rr]_{\hat{p}} && A}$$
So that $(+) T(\hat{p})\hat{p} = \< \pi_0T(\hat{p})\hat{p},\pi_1T(\hat{p})\hat{p}\>\sigma$.
Here is the calculation:
\begin{eqnarray*}
\lefteqn{T(\hat{p}\lambda)c T(\hat{p})\hat{p} } \\
& = & T(\hat{p}) \< \hat{p},p\> \mu T(\lambda) c T(\hat{p})\hat{p} \\
& = &  \< T(\hat{p})\hat{p},T(\hat{p})p\> (\lambda \x 0)m^T_\x T(\sigma\lambda) c T(\hat{p})\hat{p} \\
& = & \< T(\hat{p})\hat{p}\lambda,T(\hat{p})p0\> m^T_\x T(\lambda_2 (+)) c T(\hat{p})\hat{p} \\
& = & \< T(\hat{p})\hat{p}\lambda,T(\hat{p})p0\> m^T_\x T(\lambda_2) c_2 (+) T(\hat{p})\hat{p} \\
& = & \< T(\hat{p})\hat{p}\lambda,T(\hat{p})p0\> m^T_\x T(\lambda_2) c_2 \< \pi_0T(\hat{p})\hat{p},\pi_1T(\hat{p})\hat{p}\>\sigma \\
& = & \< T(\hat{p})\hat{p}\lambda,T(\hat{p})p0\> m^T_\x \< T(\pi_0) T(\lambda) c T(\hat{p})\hat{p},T(\pi_1) T(\lambda) c T(\hat{p})\hat{p}\> \sigma \\
& = & \< T(\hat{p})\hat{p}\lambda,T(\hat{p})p0\> m^T_\x \< T(\pi_0) T(\lambda) c T(\hat{p})\hat{p},T(\pi_1) T(\lambda) c T(\hat{p})\hat{p}\> \sigma \\
& = & \< T(\hat{p})\hat{p}\lambda T(\lambda) c T(\hat{p})\hat{p},T(\hat{p})p0 T(\lambda) c T(\hat{p})\hat{p}\>  \sigma \\
& = & \< T(\hat{p})\hat{p}\lambda T(\lambda) T(\hat{p})\hat{p},T(\hat{p})p \lambda  T(0 \hat{p})\hat{p}\>  \sigma \\
& = & \< T(\hat{p})\hat{p}\lambda \hat{p},T(\hat{p})p \lambda  T(! \zeta)\hat{p}\>  \sigma \\
& = & \< T(\hat{p})\hat{p},T(\hat{p})p \lambda  T(!) T(\zeta)\hat{p}\>  \sigma \\
& = & \< T(\hat{p})\hat{p},T(\hat{p})p \lambda  ! \zeta \>  \sigma \\
& = & T(\hat{p})\hat{p} 
\end{eqnarray*}
where $\lambda_2$ and $c_2$ are the maps defined by:
$$
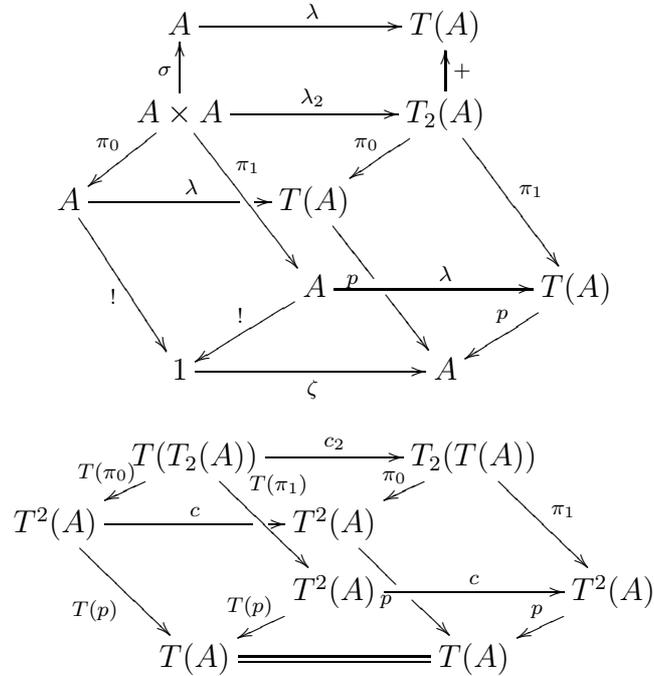
\xymatrix@=15pt{ & A \ar[rr]^\lambda & & T(A) \\
             & A \x A \ar[dl]_{\pi_0} \ar[ddr]^<<<<<<<<{\pi_1} \ar[u]^{\sigma} \ar[rr]^{\lambda_2} & & T_2(A)  \ar[dl]_{\pi_0} \ar[ddr]^{\pi_1} \ar[u]_{+} \\
            A \ar[rr]^\lambda|>>>>\hole \ar[ddr]_{!} && T(A) \ar[ddr]_<<<<<<<<p|\hole \\
            & & A \ar[rr]^\lambda \ar[dl]_{!} && T(A) \ar[dl]_p \\
            & 1 \ar[rr]_{\zeta} & & A}$$
$$\xymatrix@=7pt{ & T(T_2(A)) \ar[dl]_{T(\pi_0)} \ar[ddr]^<<<<<{T(\pi_1)}  \ar[rr]^{c_2} & & T_2(T(A))  \ar[dl]_{\pi_0} \ar[ddr]^{\pi_1}  \\
            T^2(A) \ar[rr]^c|>>>>\hole \ar[ddr]_{T(p)} && T^2(A) \ar[ddr]_<<<<<<<<p|\hole \\
            & & T^2(A) \ar[rr]^c \ar[dl]_{T(p)} && T^2(A) \ar[dl]_p \\
            & T(A) \ar@{=}[rr] & & T(A)}$$
A useful observation for the next calculation is that the following diagram commutes:
$$
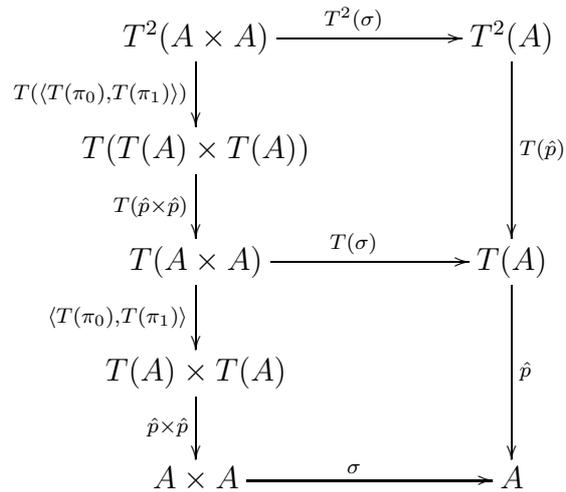
\xymatrix{T^2(A \x A) \ar[rr]^{T^2(\sigma)} \ar[d]_{T(\<T(\pi_0),T(\pi_1)\>)} && T^2(A) \ar[dd]^{T(\hat{p})} \\
                    T(T(A) \x T(A)) \ar[d]_{T(\hat{p} \x \hat{p})} \\
                     T(A \x A) \ar[rr]^{T(\sigma)} \ar[d]_{\<T(\pi_0),T(\pi_1)\>} & & T(A) \ar[dd]^{\hat{p}} \\
                     T(A) \x T(A) \ar[d]_{\hat{p} \x \hat{p}} \\
                      A \x A \ar[rr]^{\sigma} & & A}$$
It then follows that $T^2(\sigma) T(\hat{p}) \hat{p} = \< T^2(\pi_0)T(\hat{p})\hat{p},T^2(\pi_1)T(\hat{p})\hat{p} \>\sigma$.

Now we use all the above in:
\begin{eqnarray*}
&   & \lefteqn{cT(\hat{p})\hat{p}} \\
& = & T(\<\hat{p}\lambda,p0\> m^T_\x T(\sigma)) c T(\hat{p})\hat{p} \\
& = & \< T(\hat{p}\lambda),T(p0) \> m^T_\x T(m^T_\x) c T^2(\sigma) T(\hat{p})\hat{p} \\
& = & \< T(\hat{p}\lambda),T(p0) \> m^T_\x T(m^T_\x) c  \< T^2(\pi_0)T(\hat{p})\hat{p},T^2(\pi_1)T(\hat{p})\hat{p}\> \sigma \\
& = & \< T(\hat{p}\lambda),\!T(p0) \>  \<  m^T_\x T(m^T_\x) c T^2(\pi_0)T(\hat{p})\hat{p},\!m^T_\x T(m^T_\x) cT^2(\pi_1)T(\hat{p})\hat{p}\> \sigma \\
& = & \< T(\hat{p}\lambda),T(p0) \>  \<  \pi_0 c T(\hat{p})\hat{p}, \pi_1 c T(\hat{p})\hat{p}\> \sigma \\
& = & \<  T(\hat{p}\lambda) c T(\hat{p})\hat{p}, T(p0) c T(\hat{p})\hat{p}\> \sigma \\
& = & \<  T(\hat{p})\hat{p}, T(p) T(0\hat{p})\hat{p}\> \sigma \\
& = & \<  T(\hat{p})\hat{p}, T(p) T(! \zeta)\hat{p}\> \sigma \\
& = &  \<  T(\hat{p})\hat{p}, T(p) ! \zeta \> \sigma \\
& = & T(\hat{p})\hat{p}.
\end{eqnarray*}
\end{proof}

From \cite{sman3}, if $(A, \sigma_A,\zeta_A,\hat{p}_A)$ and $(B, \sigma_B,\zeta_B,\hat{p}_B)$ are differential objects then $f: A \to B$ is  said to be {\bf differentially linear} if $T(f)\hat{p}_B = \hat{p}_Af$.  We observe:

\begin{proposition}\label{linearDiffOver1}
Linear morphisms between differential bundles over $1$ are the same as maps which are differentially linear.
\end{proposition}
\begin{proof}
If $f: A \to A'$ is linear bundle morphism and $\hat{p} = \{1\}$ is the first projection for the differential bundle, then $\<\hat{p},p\>$ is the inverse of $\mu$; 
asking that $f$ be a linear bundle morphism implies by lemma \ref{mulambda} that:
	\[ \mu T(f) = (f \times f)\mu' \mbox{ or } T(f) = \<\hat{p},p\>(f \times f)\mu' \]
Post-composing by $\hat{p}$ then gives
	\[ T(f)\hat{p} = \<\hat{p}f,pf\>\mu'\hat{p} = \<\hat{p}f,pf\>\pi_0 = \hat{p}f. \]
Conversely, suppose we have a differentially linear map; that is, $T(f)\hat{p} = \hat{p}f$.  Then the above proves that
	\[ \mu T(f) = (f \times f)\mu' \]
on the first component, and on the second component, asking $T(f)p = \<\hat{p},pf\>\mu p = \<\hat{p},pf\>\pi_1 =  pf$ is just naturallity of $p$.  However, we also have that $f$ preserves the zeroes of the bundles, since
	\[ f0 = 0T(f) \Rightarrow f0\hat{p} = 0T(f)\hat{p} \Rightarrow ! \zeta' = ! \zeta f \]
so since $!$ is epic, $\zeta' = \zeta f$.  Thus by lemma \ref{mulambda}, $f$ is a linear differential bundle morphism.  
\end{proof}


\subsection{Differential objects in representable tangent categories}\label{sectionRepDiffObjects}

One aspect of differential objects that was not investigated in \cite{sman3} was what form they take when the tangent category is representable; that is, when there is an object $D$ for which $TM = M^D$.  Section 5 of \cite{sman3} discusses what conditions on an object $D$ are required so that defining $TM = M^D$ gives a tangent category.  In particular, $D$ has:
\begin{itemize}
	\item maps $\odot: D \x D \to D$ and $\wp: 1 \to D$ (``comultiplication'' and ''zero'');
	\item the pushout of $1 \to^{\wp} D$ along itself exists; this pushout will be denoted by $D \star D$, with injections $\imath_0: D \to D \star D, \imath_1: D \to D \star D$;
	\item a map $\delta: D \to D \star D$.  
\end{itemize}
satisfying various axioms (see definition 5.6 of \cite{sman3}).  Such a $D$ is referred to as an ``infinitesimal'' object.    

Now, suppose that we have a differential object $A$ in a representable tangent category with corresponding infinitesimal object $D$.  By the previous section, $A$ is a differential bundle over 1 and hence has a lift map $\lambda: A \to TA$.  But since $TA = A^D$, this is equivalent to giving an ``action'' $D \times A \to A$.  The following theorem characterizes the differential object axioms in terms of this action, written in the term logic used in section 5 of \cite{sman3}.  

\begin{theorem}
Suppose that $(\X,\T)$ is a representable tangent category with infinitesimal object $D$.  Then to give a differential object is equivalent to giving an ``infinitesimal module'': a commutative monoid $(A,\sigma,\zeta)$ with a map $\act: D \times A \to A$ so that:
\begin{enumerate}[(i)]
	\item $\wp \act a = \zeta$.
	\item $d \act \zeta = \zeta$.
	\item $d \act \sigma(a_1,a_2) = \sigma(d \act a_1, d \act a_2)$.
	\item $\sigma(d \act a_1, d \act a_2) = \left\{ \begin{array}[c]{lcl} \imath_0(d) & \mapsto & d \act a_1 \\
                                                             \imath_1(d) & \mapsto & d \act a_2
                                       \end{array} \right\} \delta(d)$
	\item $d_1 \act (d_2 \act a) = (d_1 \odot d_2) \act a$.
	\item The map $(a_1,a_2) \mapsto \lambda d.\sigma(d \act a_1,a_2)$ is invertible (ie., for each $f: D \to A$, there exists unique $a,b \in A$ so that $f(d) = db + a$).  
\end{enumerate}
\end{theorem}
\begin{proof}
The result is a straightforward application of uncurrying the differential bundle axioms.
\end{proof}

If the representable tangent category is in fact a model of SDG, then more can be said.  Recall that in SDG one has a ``line object $R$'', and it is often useful to consider ``Euclidean $R$-modules'': $R$-modules $A$ with the property that for any $f: D \to A$ there exists unique $b \in A$ such that $f(b) = db + a$ (for example, see \cite{lavendhomme}, pg. 5).  
	
\begin{theorem}\label{diffBundlesInSDG}
Suppose that $\X$ is a model of SDG with line object $R$.  Then to give a differential object in the corresponding tangent category $(\X,\T)$ is precisely to give a Euclidean $R$-module. 
\end{theorem}
\begin{proof}
Let $D$ be the corresponding infinitesimal object of $R$.  By the theorem above, we need to show how to translate between an infinitesimal module and an $R$-module.  Recall that in SDG the object $R$ is the pullback
\[
\bfig
	\square[R`D^D`1`D;i``p`\wp]
\efig
\]
and there is an inclusion $j$ of $D$ into $R$ (which is given by the unique map $D \to R$ given by currying the multiplication $\odot: D \times D \to D$ and using the above pullback description of $R$).  

Now, given an infinitesimal module $A$ we can define an action of $R$ by the composite
	\[ R \times A \to^{^{i \times \lambda}} D^D \times A^D \to^{\circ} A^D \to^{\hat{p}} A \]
Conversely, given a Euclidean $R$-module $A$ with action $m: R \times A \to A$, we can define an action of $D$ on $R$ by
	\[ D \times A \to^{j \times 1} R \times A \to^m A \]
It is now straightforward to check that the axioms for an infinitesimal module translate directly to the axioms for a Euclidean $R$-module and vice versa, with one exception: axiom (iv) for an infinitesimal module.  This axiom is equivalent to the differential bundle axiom 
\[
\bfig
	\square[A \times A`A`T_2A`TA;\sigma`\<\pi_0 \lambda,\pi_1 \lambda\>`\lambda`+]
\efig
\]
This does not come directly from any axiom for a Euclidean $R$-module, but in fact is automatic given the other axioms.  Indeed, while in an arbitrary tangent category $T_2(A)$ is given by $A^{D \star D}$, in a model of SDG, $T^2(A)$ is also given by $A^{D(2)}$, where $D(2) = \{(d_1,d_2) \in D \times D: d_1d_2 = 0\}$.  In this case, the map $\<\pi_0 \lambda,\pi_1 \lambda\>$ is given by the map
	\[ (a_1,a_2) \mapsto \lambda(d_1,d_2).\sigma(d_1a_1,d_2a_2) \]
since post-composing this map by the projections $\pi_0, \pi_1: T_2A \to TA$ gives $\pi_0 \lambda$ and $\pi_1 \lambda$.  Then since $+$ is given by $\delta: D \to D(2): d \mapsto (d,d)$, $\<\pi_0 \lambda,\pi_1 \lambda\>+$ is simply
	\[ (a_1,a_2) \mapsto \lambda d.\sigma(da_1,da_2) \]
which then by (iii) equals $d.(d \act \sigma(a_1,a_2))$, as required.  
\end{proof}


\subsection{Cartesian differential categories}\label{secCartDiffCats}

Recall that a Cartesian differential category \cite{cartDiff} is a left-additive category\footnote{These are examples of so called {\em skew\/} enriched categories following \cite{Street}.  Concretely, this means each homset is a commutative monoid and with respect to composition satisfies the left distribution law -- $f(h+k) = fh+fk$ and $f0 = 0$ -- but not necessarily the right distribution law.} with a combinator 
$$\infer{X \x X \to_{D(f)} Y}{X \to^f Y}$$
(called ``differentiation'') satisfying certain axioms.  Here, we give an alternative form of the axioms as explained in proposition 4.2 of \cite{sman3}.  This form of the axioms allows one to align tangent structure more conveniently with differential structure.  
\begin{enumerate}[{\bf [CD.1]}]
\item $D(f+g) = D(f)+D(g)$ and $D(0)=0$;
\item $\< a+b,c\>D(f) = \< a,c\>D(f) + \< b,c\>D(f)$ and $\< 0,a\>D(f) = 0$;
\item $D(\p_0) = \p_0\p_0$, and $D(\p_1)= \pi_0\p_1$;
\item $D(\<f,g\>) = \< D(f),D(g)\>$;
\item $D(fg) = \<D(f),\p_1f \>D(g)$;
\item $\<\<a,0\>,\<0,d\>\>D(D(f)) = \<a,d\>D(f)$;
\item $\< \< a,b\>,\< c,d\>\> D(D(f)) = \<\< a,c\>,\< b,d\>\> D(D(f))$;
\end{enumerate}

As noted earlier, a Cartesian differential category is always a tangent category with $T(A) := A \x A$ and $T(f) = \<D(f),\pi_1f\>$.  Furthermore, it is straightforward to check that every object in a Cartesian differential category is, with respect to the  structure ${\sf A} := ( !_A, 0,\pi_0+\pi_1, \< 1,0\>)$, 
canonically a differential bundle over the final object.   For this, in particular, one must check the universal property of the lift, that the diagrams $(1)$,$(2)$,$(3)$ and $(6)$ in the proof of Proposition 
\ref{diffObjeqdiffBun} commute, and the coherence with the lift is satisfied. 

The results of the previous section imply that when the lift is preserved by $f: A \to B$ then $f$ should be linear in the sense that $D[f]= \pi_0f$.  
For a Cartesian differential category, we may verify this with a direct proof.  The preservation of the lift tells us 
$$\< 1,0\> T(f) = \< \<1,0\>D[f],0f\> = \<f,0\>$$ 
implying that $0f = 0$ and $\< 1,0\>D[f] = f $. 
We want to prove that $D[f]= \pi_0f$.  As we see below, it actually suffices to know that $\< 1,0\>D[f] = f$. Here is the calculation:
\begin{eqnarray*}
D[f] & = & D[\<1,0\>D[f]] = \<D[\<1,0\>],\pi_1\<1,0\>\>D[D[f]] \\
& = & \<\<\pi_0,0\>,\<\pi_1,0\>\>D[D[f]] = \< \pi_0,0 \>D[f] \\
& = & \pi_0\< 1,0 \>D[f] = \pi_0f
\end{eqnarray*}


\subsection{Coherent differential structure}

To determine, conversely, whether a Cartesian tangent category is, in fact, a Cartesian differential category one must clearly be able to assign to each object a ``canonical'' differential bundle structure.   There is no reason, however, why an arbitrary assignment of such structures will allow one to reconstruct a Cartesian differential structure for the category which is compatible with the existing tangent structure.  Clearly some additional compatibility is required to ensure the choice of bundle structure at each object is compatible with the Cartesian tangent structure.  

\begin{definition}\label{defnCoherentDiffStructure}
A Cartesian tangent category has {\bf coherent differential structure} if every object $A$ has an associated structure as a differential bundle over $1$,  
${\sf A} = (!_A,\zeta_A,\sigma_A,\lambda_A)$, which we refer to as its {\bf canonical} differential structure, such that:
\begin{enumerate}[{\bf [CDS.1]}]
\item the canonical differential structure on $A \x B$ is the product of the canonical structures of its components:
           $$\hspace{-0.75cm}{\sf A} \x {\sf B} =\left( \begin{array}[c]{l} !:A \x B \to 1, \\
                                                    \zeta_{A\x B} = \<\zeta_A,\zeta_B\>: 1\to A\x B, \\
                                                    \sigma_{A\x B}= {\sf ex}(\sigma_A \x \sigma_B): (A \x B) \x (A \x B) \to A \x B, \\
                                                    \lambda_{A \x B} = \< \lambda_A,\lambda_B \>: A \x B \to T(A\x B) \end{array} \right)$$
\item the canonical differential structure on $T(A)$ is given by:
            $$T({\sf A}) = \left( \begin{array}[c]{l} !:T(A) \to 1, \\
                                                     \zeta_{T(A)} = m^T_1T(\zeta_A): 1 \to T(A), \\
                                                     \sigma_{T(A)} = m_\x^TT(\sigma_A): T(A) \x T(A) \to T(A), \\
                                                     \lambda_{T(A)} = T(\lambda_A)c: T(A) \to T^2(A) \end{array} \right)$$
\end{enumerate}
where $m^T_1:= !_{T(1)}^{-1}$ and $m_\x^T := \<T(\pi_0),T(\pi_1)\>^{-1}$.  
\end{definition}

Our first objective is to translate these conditions into equivalent requirements for differential (object) structure.  Preliminary to this 
we prove two useful equations which we shall use repeatedly in the proofs below:

\begin{lemma} \label{muproj}
In a Cartesian tangent category with a coherent differential structure 
	\[ \mu_{A \x B} T(\pi_0) = (\pi_0 \x \pi_0) \mu_A \mbox{ and } \mu_{A \x B} T(\pi_1) = (\pi_1 \x \pi_1) \mu_B. \]
\end{lemma}

\begin{proof} We shall prove the first:
\begin{eqnarray*}
\mu_{A \x B} T(\pi_0) & = & ((\lambda_A \x \lambda_B) m^T_\x \x 0) m^T_\x T({\sf ex}(\sigma_A \x \sigma_B)) \pi_0 \\
& = & ((\lambda_A \x \lambda_B) m^T_\x \x 0) m^T_\x T(\pi_0 \x \pi_0) T(\sigma_A) \\
& = & (((\lambda_A \x \lambda_B) m^T_\x T(\pi_0)) \x (0T(\pi_0))) m^T_\x  T(\sigma_A) \\
& = & ((\pi_0 \lambda_A) \x (\pi_0 0)) m^T_\x  T(\sigma_A) \\
& = & (\pi_0 \x \pi_0) (\lambda_A \x 0) m^T_\x T(\sigma_A) \\
& = & (\pi_0 \x \pi_0) \mu_A
\end{eqnarray*}
the other identity is similar.  
\end{proof}

\begin{proposition} \label{candiff}
A Cartesian tangent category has coherent differential structure if and only if each object has an associated differential structure such that
\begin{enumerate}[{\bf [CDS.1]$^\prime$}]
\item $A \x B$ has coherent differential structure satisfying the following equations:
       $$\zeta_{A\x B} \!=\! \<\zeta_A,\!\zeta_B\>, \sigma_{A\x B}\!= \!{\sf ex}(\sigma_A \x \sigma_B), \hat{p}_{A \x B} \!=\! \<T(\pi_0)\hat{p}_A,\!T(\pi_1)\hat{p}_B \>$$
\item $T(A)$ has coherent structure satisfying the following equations:
       $$\zeta_{T(A)} = m^T_1 T(\zeta_A),~~ \sigma_{T(A)} = m^T_\x T(\sigma_A), ~~ \hat{p}_{T(A)} = cT(\hat{p}_A)$$
\end{enumerate}
\end{proposition}

\begin{proof}
We shall check only the second condition, leaving the first to the reader.  
\begin{description}
\item{\bf [CDS.2]$^\prime$}   
We need to show that for a differential object 
$\hat{p}_{T(A)} = cT(\hat{p}_A)$ if and only if $\lambda_{T(A)} = T(\lambda_A)c$.  Starting with a differential object 
we define $\lambda_{T(A)}$ as 
$$
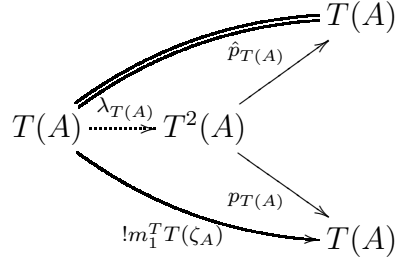
\xymatrix{& & T(A) \\
            T(A) \ar@/_1pc/[rrd]_{!m^T_1T(\zeta_A)} \ar@{=}@/^1pc/[rru] \ar@{..>}[r]^{\lambda_{T(A)}} & T^2(A) \ar[ru]^{\hat{p}_{T(A)}} \ar[rd]_{p_{T(A)}} \\ 
            & & T(A)}$$
We must check that $T(\lambda_A)c$ works:
\begin{eqnarray*}
T(\lambda_A)c\hat{p}_{T(A)} & = & T(\lambda_A)ccT(\hat{p}_{A}) = T((\lambda_A\hat{p}_{A}) = T(1_A) = 1_{T(A)} \\
T(\lambda_A)cp_{T(A)} &= & T(\lambda_A)T(p_A)= T(\lambda_Ap_A) = T(!_A\zeta_A) = !m^T_1T(\zeta_A).
\end{eqnarray*}
Conversely for a differential bundle over $1$ we must check $ \hat{p}_{T(A)} = \{ 1_{T^2(A)} \} = cT(\{ 1_{T(A)}\}) = c T(\hat{p}_A)$.  Using \ref{lemmaBracketT} we have:
$$cT(\{ 1_{T(A)} \}) = c \{ T(1_{T(A)})c \} = c \{ c\} = \{ cc\} = \{ 1_{T^2(A)} \}.$$
\end{description}
\end{proof}

It is straightforward now to observe: 

\begin{lemma}
Cartesian differential categories always have a coherent differential structure given by ${\sf A} = (!_A,+,0,\<1,0\>)$. 
\end{lemma}

The main theorem of the section is:

\begin{theorem} \label{canonical-thm}
Every Cartesian tangent category with a coherent differential structure is a left additive category and has an associated  differential
$$\infer{D[f]:=  \mu_A T(f) \hat{p}_B: A \x A \to B}{f: A \to B}$$
which makes it a Cartesian differential category.
\end{theorem}

Before proving this we establish an identity which plays a key role in proving that a tangent category with coherent differential structure satisfies \cd{7}.

\begin{lemma}\label{exchangeForCanonicalDiff}
In a Cartesian tangent category with canonical differential structure the following equation holds for all objects $A$:
	 $${\sf ex} \mu_{A \x A} T(\mu_A) = \mu_{A \x A}T(\mu_A) c_A.$$
\end{lemma}
\begin{proof}
To prove this we post-compose each side of the equation with the projections
$T(\hat{p}_A)\hat{p}_A$, $T(\hat{p}_A)p_A$, $T(p_A)\hat{p}_A$, and $T(p_A)p_A$ and show their projections are equal, by using the result that $cT(\hat{p})\hat{p} = T(\hat{p})\hat{p}$:
\begin{description}
\item{$T(\hat{p}_A)\hat{p}_A$:}~
\begin{eqnarray*}
{\sf ex}\mu T(\mu)T(\hat{p})\hat{p} & = & {\sf ex}\mu T(\pi_0)\hat{p} = {\sf ex}(\pi_0 \x \pi_0)\mu\hat{p} \\
& = & {\sf ex}(\pi_0 \x \pi_0)\pi_0 =\pi_0\p_0 \\
\mu T(\mu)c T(\hat{p}) \hat{p} & = &  \mu T(\mu) T(\hat{p}) \hat{p} = \mu T(\pi_0) \hat{p} \\
& = &  (\pi_0 \x \pi_0) \mu \hat{p} = (\pi_0 \x \pi_0) \pi_0 = \pi_0\pi_0 
\end{eqnarray*}
\item{$T(\hat{p}_A)p_A$:}
\begin{eqnarray*}
{\sf ex}\mu T(\mu)T(\hat{p})p & = & {\sf ex}\mu T(\pi_0)p = {\sf ex}(\pi_0 \x \pi_0)\mu p \\
& = & {\sf ex}(\pi_0 \x \pi_0)\pi_1 = {\sf ex}\pi_1\p_0 = \pi-0\pi_1 \\
\mu T(\mu)c T(\hat{p}) p & = &  \mu T(\mu) c p \hat{p} = \mu T(\mu) T(p) \hat{p} = \mu T(\pi_1) \hat{p} \\
& = &  (\pi_1 \x \pi_1) \mu \hat{p} = (\pi_1 \x \pi_1) \pi_0 = \pi_0\pi_1 
\end{eqnarray*}
\item{$T(p_A)\hat{p}_A$:}
\begin{eqnarray*}
{\sf ex}\mu T(\mu)T(p)\hat{p} & = & {\sf ex}\mu T(\pi_1)\hat{p} = {\sf ex}(\pi_1 \x \pi_1)\mu \hat{p} \\
& = & {\sf ex}(\pi_1 \x \pi_1)\pi_0 = {\sf ex}\pi_0\p_1 = \pi_1\pi_0 \\
\mu T(\mu)c T(p)\hat{p} & = &  \mu T(\mu) p \hat{p} = \mu p \mu \hat{p} = \pi_1 \pi_0
\end{eqnarray*}
\item{$T(p_A)p_A$:}
\begin{eqnarray*}
{\sf ex}\mu T(\mu)T(p)p & = & {\sf ex}\mu p \mu p = {\sf ex} \pi_1 \pi_1 = \pi_1 \pi_1  \\
\mu T(\mu)c T(p)p & = &  \mu T(\mu) p p = \mu p \mu p = \pi_1 \pi_1
\end{eqnarray*}
\end{description}

\end{proof}

\begin{proof} (Of theorem \ref{canonical-thm})
The canonical differential structure makes each object canonically a commutative monoid which implies that the category is left additive (see 1.2.2 of \cite{cartDiff}).  
It remains therefore to show that the differential so defined satisfies {\bf [CD.1]}--{\bf [CD.7]}.  For this, first note that
$$D_A[f] = \mu_A T(f) \hat{p}_B = \mu_A T(f) \{ 1_{T(B)} \} = \{ \mu_A T(f). \}$$
We now check each of the axioms:
\begin{description}
\item{\bf [CD.3]}
$D[1] = \{ \mu T(1) \} = \{ \mu \} = \pi_0$ and $D[\pi_i] = \{\mu T(\pi_i)\} = \{ \mu\} \pi_i = \pi_0 \pi_i$ where the penultimate step uses 
the linearity of $(T(\pi_i),1_M): {\sf q}_2 \to {\sf q}$.
\item{\bf [CD.4]}  We have the following calculation:
\begin{eqnarray*}
D[\<f,g\>] & = & T( \<f,g\> )\hat{p}_{A \x B} = T(\< f,g\>) m^T_{A \x B} (\hat{p}_A \x \hat{p}_B) \\
           & = & \< T(f),T(g) \> (\hat{p}_A \x \hat{p}_B) = \< T(f)\hat{p}_A,T(g)\hat{p}_B \> \\
           & = &  \< D[f],D[g]\>
\end{eqnarray*}
\item{\bf [CD.1]}
\begin{eqnarray*}
 D[f + g] & = & \{ \mu T(\<f,g\>\sigma_A) \} \\
              & = & \{ \mu T(\<f,g\> ) \} \sigma_A ~~\mbox{($(\sigma,1_A): A\x A \to A$ is linear)} \\
              & = & D[\<f,g\>] \sigma_A = \< D[f],D[g] \> \sigma_A \\
              & = & D[f] + D[g]
\end{eqnarray*}
\item{\bf [CD.2]}  We have:
\begin{eqnarray*}
\lefteqn{\< f + g,h\>D[f]} \\ & = & \<\< f,g\>\sigma_A,h\> \mu T(f)\hat{p}_A \\
& = & \<\< f,g\>\sigma_A \lambda_A,h 0\>m^T_\x T(\sigma) T(f)\hat{p}_A \\
& = & \<\< f \lambda_A,g \lambda_A\>+,\< h 0,h0\> +\>m^T_\x T(\sigma) T(f)\hat{p}_A \\
& = & \< \< f \lambda_A,h0\>m^T_\x T(\sigma),\<g \lambda_A,h0\>m^T_\x T(\sigma) \> + T(f) \hat{p}_A \\
& & ~~~~~~~~\mbox{(exchange of additions - Lemma \ref{lemmaInterchange})}\\
& = & \< \< f \lambda_A,h0\>m^T_\x T(\sigma) T(f),\<g \lambda_A,h0\>m^T_\x T(\sigma) T(f) \> + \hat{p}_A \\
& & ~~~~~~~~\mbox{(naturality of $+$)}\\
& = &  \< \< f \lambda_A,h0\>m^T_\x T(\sigma) T(f)\hat{p}_A,\<g \lambda_A,h0\>m^T_\x T(\sigma) T(f)\hat{p}_A \> \sigma \\
& & ~~~~~~~~\mbox{($\hat{p}$ is additive)} \\
& = &  \< \< f,h\>\mu T(f)\hat{p}_A,\<g h\>\mu T(f)\hat{p}_A \> \sigma \\
& = &  \<f,h\>D[f] + \<g,h\> D[g]
\end{eqnarray*}
\item{\bf [CD.5]}  We have:
\begin{eqnarray*}
D[fg] & = & \mu T(fg)\hat{p}_C \\
      & = & \mu T(f)\< \hat{p}_B,p_B \> \mu T(g) \hat{p}_C \\
      & = & \< \mu T(f)\hat{p}_B,\mu T(f)p_B \> D[g] \\
      & = & \< D[f],\mu p_A f \> D[g] = \< D[f],\pi_1f\> D[g] 
\end{eqnarray*}
\item{\bf [CD.6]} We shall use the equality:
$$\<x,\zeta_A\>,\<\zeta_A,y\>\> \mu_{A \x A}T(\mu_A) = \<x,y\> \mu_A \ell$$
in the following:
\begin{eqnarray*}
\<\<x,0\>,\<0,y\>\>D[D[f]] & = &   \<x,\zeta_A\>,\<\zeta_A,y\>\> \mu_{A \x A}T(\mu_A)T(T(f)\hat{p}_B)\hat{p}_B \\
& = & \< x,y \> \mu_A \ell_A T^2(f)T(\hat{p}_A)\hat{P}_A \\
& = &  \< x,y \> \mu_A T(f) \ell_AT(\hat{p}_A)\hat{P}_A \\
& = & \< x,y \> \mu_A T(f) \hat{p}_A = \<x,y\>D[f]
\end{eqnarray*}
It remains to prove the equality which we do by post-composing each side with the projections 
$T(\hat{p}_A)\hat{p}_A$, $T(\hat{p}_A)p_A$, $T(p_A)\hat{p}_A$, and $T(p_A)p_A$:
\begin{description}
\item{$T(\hat{p}_A)\hat{p}_A$:}
\begin{eqnarray*}
\lefteqn{\< \<x,\zeta_A\>,\<\zeta_A,y\>\> \mu_{A \x A}T(\mu_A)T(\hat{p}_A)\hat{p}_A} \\ 
& = & \<x,\zeta_A\>,\<\zeta_A,y\>\> \mu_{A \x A}T(\pi_0)\hat{p}_A \\
& = & \<x,\zeta_A\>,\<\zeta_A,y\>\> (\p_0 \x \pi_0) \mu_{A}\hat{p}_A \\
& = & \<x,\zeta_A\>,\<\zeta_A,y\>\> (\p_0 \x \pi_0) \pi_0 \\
& = & \<x,\zeta_A\>,\<\zeta_A,y\>\> \pi_0 \pi_0 = x \\
\lefteqn{\<x,y\> \mu_A \ell T(\hat{p}_A)\hat{p}_A} \\
& = & \< x,y\>\mu_A \hat{p}_A \\
& = & \< x,y\> \pi_0 = x 
\end{eqnarray*}
\item{$T(\hat{p}_A)p_A$:}
\begin{eqnarray*}
\lefteqn{\<\<x,\zeta_A\>,\<\zeta_A,y\>\> \mu_{A \x A}T(\mu_A)T(\hat{p}_A)p_A} \\
& = & \<x,\zeta_A\>,\<\zeta_A,y\>\> \mu_{A \x A}p_{A\x A} \mu_A \hat{p}_A \\
& = & \<x,\zeta_A\>,\<\zeta_A,y\>\> \pi_1 \pi_0 = \zeta_A \\
\lefteqn{ \<x,y\> \mu_A \ell T(\hat{p}_A)p_A} \\
& = &  \<x,y\> \mu_A \ell p_{T(A)} \hat{p}_A \\
& = & \<x,y\> \mu_A p_A 0_A \hat{p}_A \\
& = & \<x,y\> \mu_A p_A ! \zeta_A = \zeta_A
\end{eqnarray*}
\item{$T(p_A)\hat{p}_A$:}
\begin{eqnarray*}
\lefteqn{\<\<x,\zeta_A\>,\<\zeta_A,y\>\> \mu_{A \x A}T(\mu_A)T(p_A)\hat{p}_A} \\
& = & \<x,\zeta_A\>,\<\zeta_A,y\>\> \mu_{A \x A}T(\pi_1)\hat{p}_A \\
& = & \<x,\zeta_A\>,\<\zeta_A,y\>\> (\pi_1,\pi_1)\mu_{A}\hat{p}_A \\
& = & \<x,\zeta_A\>,\<\zeta_A,y\>\> (\pi_1,\pi_1)\pi_0 \\
& = & \<x,\zeta_A\>,\<\zeta_A,y\>\> \pi_0 \pi_1 = \zeta_A \\
\lefteqn{\<x,y\> \mu_A \ell T(p_A)\hat{p}_A} \\
& = & \< x,y\> \mu p 0 \hat{p} \\
& = & \< x,y\> \mu p ! \zeta_A = \zeta 
\end{eqnarray*}
\item{$T(p_A)p_A$:}
\begin{eqnarray*}
\lefteqn{\< \<x,\zeta_A\>,\<\zeta_A,y\>\> \mu_{A \x A}T(\mu_A)T(p_A)p_A }\\
& = & \<x,\zeta_A\>,\<\zeta_A,y\>\> \mu_{A \x A}p_{A \x A} \mu_A p_A \\
& = & \<x,\zeta_A\>,\<\zeta_A,y\>\> \pi_1 \pi_1 = y \\
\lefteqn{\<x,y\> \mu_A \ell T(p_A)p_A} \\
 & = & \<x,y\> \mu_A p_A 0_A p_A \\
& = & \<x,y\> \pi_1 = y
\end{eqnarray*}
\end{description}

\item{\bf [CD.7]} We want  $D[D[f]] = {\sf ex} D[D[f]]$:
\begin{eqnarray*}
D[D[f]] & = & \mu_{A \x A} T(\mu_AT(f)\hat{p}_A)\hat{p}_A \\
        & = & {\sf ex} \mu_{A \x A} T(\mu_A) c T^2(f) T(\hat{p}_A)\hat{p}_A \mbox{ (by Lemma \ref{exchangeForCanonicalDiff})} \\
        & = & {\sf ex} \mu_{A \x A} T(\mu_A) T^2(f) c T(\hat{p}_A)\hat{p}_A \\
        & = & {\sf ex} \mu_{A \x A} T(\mu_A) T^2(f) T(\hat{p}_A)\hat{p}_A \mbox{ (by Lemma \ref{propCCondition})} \\
        & = & {\sf ex} \mu_{A \x A} T(\mu_A T(f)\hat{p}_A)\hat{p}_A \\
        & = & {\sf ex} D[D[f]]
\end{eqnarray*}
\end{description}

\end{proof}


\section{Transverse and display systems}\label{secTransAndDisplay}


The theory of transverse maps in classical differential geometry -- as embodied by smooth manifolds -- indicates that the issue of which pullbacks exist and are preserved by the tangent functor can be quite delicate.  In the general setting of tangent categories the question of which pullbacks exist is, as a consequence, also important.  In particular, considering Cartesian differential categories as a source of examples of Cartesian tangent categories, we cannot require many pullbacks -- outside those implied by having products -- to exist.  Yet, the presence of pullbacks is clearly important in the development of the fibrational properties of tangent categories.  Thus, the purpose of this section is formalize these issues in the context of arbitrary tangent categories.  

\subsection{Transverse systems}

A transverse system on a category is a way of specifying the existence of certain multiple pullbacks and facilitates formalizing the requirement that functors should preserve these pullbacks.   Let ${\mathbf 2}$ denote the arrow category $0 \to 1$.  

\begin{definition} 
A {\bf transverse system} ${\cal Q}$ on a category $\X$ is a graded collection of functors $\left( {\cal Q}_n \right)_{n \in \N}$ with each $q \in {\cal Q}_n$ a pullback preserving functor $q: {\mathbf 2}^n \to \X$ such that:
\begin{enumerate}[(a)]
	\item  All functors $A: {\mathbf 2}^0 \to \X$ (objects) and $f: {\mathbf 2}^1  \to \X$ (maps) are in ${\cal Q}$.  
	\item If $r: {\mathbf 2}^m \to {\mathbf 2}^n$ is pullback preserving and $q: {\mathbf 2}^n \to \X$ is in ${\cal Q}$ then $rq$ is in ${\cal Q}$. 
	\item If $q: {\mathbf 2}^n \to \X$ is in ${\cal Q}$ and $\alpha: q \Rightarrow q':  {\mathbf 2}^n \to \X$ is a natural isomorphism then $q'$ is in ${\cal Q}$.  
	\item For any $n \geq 3$, if $q: {\mathbf 2}^{n} \to \X$ has all its faces in ${\cal Q}$ then it is itself in ${\cal Q}$, where the faces $d_0^{i,n}$ and $d_1^{i,n}$ are functors defined by 
		$$ d_0^{i,n}: {\mathbf 2}^n \to {\mathbf 2}^{n+1}; (v_1,...,v_i,..,v_n) \mapsto (v_1,...,0,v_i,..,v_n)  ~~~\mbox{and}$$
		$$~~~d_1^{i,n}: {\mathbf 2}^n \to {\mathbf 2}^{n+1}; (v_1,...,v_i,..,v_n) \mapsto (v_1,...,1,v_i,..,v_n).$$
\end{enumerate}
Each functor $q \in {\cal Q}$ is called a ${\cal Q}$-{\bf transverse} and determines a finite multiple pullback in $\X$: the base of the pullback is $q(1,...,1)$, the apex is $q(0,..,0)$, and the arrows of the pullback diagram are generated by the maps $$q((1,..,0,...,1) < (1,..,1,...,1).$$

\end{definition}

Condition (b) allows degenerate transverses which using (c) can be made trivially non-degenerate: in this manner pullbacks along isomorphisms, for example, are always included. Condition (c) also ensures that equivalent pullbacks are always included.
Condition (d) ensures that a transverse system is generated by its pullback squares.  Of course, while this is the case it still may be more convenient to specify the system using multiple pullbacks -- which is the case, in particular, for tangent categories.

Clearly, a transverse system need by no means specify every multiple pullback which exists in the category.   Indeed, the point of a transverse system is to allow one to be parsimonious with limits and the preservation thereof.

It is clear that given an arbitrary collection ${\cal P}$ of pullback preserving functors $p: {\mathbf 2}^n \to \X$, one can close this collection under the rules given above to get a transverse system $\widehat{{\cal P}}$.  We record this formally in the following result:

\begin{lemma} If ${\cal P}$ is any collection of pullback preserving functors from ${\mathbf 2}^n$ to $\X$ then 
\begin{enumerate}[(i)]
\item Every functor in $\widehat{{\cal P}}$ is indeed a pullback preserving functor and so $\widehat{{\cal P}}$ is a transverse system;
\item $\widehat{\widehat{{\cal P}}} = \widehat{{\cal P}}$;
\item Whenever a functor $F: \X \to \Y$ preserves all elements of ${\cal P}$ then it will necessarily preserve all elements of  $\widehat{{\cal P}}$.  
\end{enumerate}
\end{lemma}

Notice that in the definition of a tangent category we have explicitly required the tangent functor to preserve the $n$-fold pullbacks of projections and the pullbacks for the universality of the lift (as noted earlier, however, this second requirement is automatic given the other axioms in a tangent category).  Thus, we make the following definition:

\begin{definition}
A \textbf{tangent category with a transverse system} consists of a tangent categoy $(\X,\T)$ with a transverse system $\cal Q$ such that:
\begin{itemize}
	\item the pullback powers of $p$ are in $\cal Q$;
	\item the pullbacks determined by the university of the lift are in $\cal Q$;
	\item $T$ preserves $\cal Q$-transverses.  
\end{itemize}
\end{definition}

\begin{example}  {\em ~
\begin{enumerate}[{\em (i)}]
	\item For any tangent category there is a basic transverse system generated by the pullback powers of projections $p:T(M) \to M$, the pullbacks in the universality of the lift, and these pullbacks' images under $T^n$.

	\item In the category of smooth manifolds a key transverse system of interest is generated by the pullbacks of a pair of transverse maps (see page 203 of \cite{lee}).  
	
	\item In a model of SDG all pullbacks exist and are preserved by the tangent functor, and so the standard transverse system on a model of SDG is simply all pullbacks.
\end{enumerate} }
\end{example}

\begin{definition}
A functor $F: (\X,{\cal Q}) \to (\X' ,{\cal Q}')$ between categories with specified transverse systems is {\bf transverse} if it sends $Q$-tranverses to ${\cal Q}'$-transverses.  
\end{definition}

Note that by (c) in the previous lemma, if ${\cal Q} = \widehat{{\cal P}}$, then to check that $F$ is transverse it suffices to check that $F$ sends elements of ${\cal P}$ to elements of ${\cal Q}'$.   

\begin{example} {\em ~
\begin{enumerate}[{\em (i)}]

\item In a tangent category with the minimal transverse system, the functors $T^n$ and $T_n$ are transverse functors.

\item ``Well-adapted models'' of SDG consist of a functor $F$ from the category of smooth manifolds to a model of SDG $\mb{E}$ which in particular sends transverse pullbacks (in the classical sense) to pullbacks in $\mb{E}$ (see \cite{kock}, page 142) and hence can be seen as a transverse functor between these categories with their transverse systems as described above.

\end{enumerate} }

\end{example}


\subsection{Display systems}\label{Display}

When considering fibrations, it will be useful to consider maps along which all pullbacks exist and are in the transverse system.  

\begin{definition}
If ${\cal Q}$ is a transverse system on a category, we say a map $f: A \to B$ in the category is {\bf ${\cal Q}$-quarrable} if every pullback along $f$ exists and is in ${\cal Q}$.  
\end{definition}

This means that $f$ is a ``quarrable'' map\footnote{The terminology  ``quarrable map'', developed in S\'eminaire de G\'eom\'etrie Alg\'ebrique, means a map along which all pullbacks exist.  Such a map is also sometimes called ``carrable''.  ``Quarrable'' is the French word for ``quadrable'' while ``carrable'', translated literally, is ``squarable''.   Quarrable maps were called ``display'' maps by Paul Taylor, who had in mind their application to fibrations: we shall adopt this terminology too.}  in the usual sense: namely that all pullbacks along that map exist.  Of course, to be ${\cal Q}$-quarrable requires, in addition, that all these pullbacks are in the transverse system ${\cal Q}$.  It is clear that: 

\begin{lemma}
In any category with a transverse system ${\cal Q}$:
\begin{enumerate}[ (i) ] 
\item  Isomorphisms are always ${\cal Q}$-quarrable;
\item  A map $f$ is ${\cal Q}$-quarrable if and only if every transverse with apex the codomain of $f$ is a face of a transverse in $\widehat{{\cal Q}}$ with $f$ added to the multiple pullback.  
\end{enumerate}
\end{lemma}

A transverse functor $F: (\X,{\cal Q}) \to (\X' ,{\cal Q}')$ will not necessarily send a ${\cal Q}$-quarrable map to a ${\cal Q}'$-quarrable map.  Sometimes, however, one would like to demand that certain ${\cal Q}$-quarrable maps should be preserved in this sense.     This can be arranged most conveniently by associating a display system, ${\cal D}$, of ${\cal Q}$-quarrable maps in the category.

\begin{definition} \label{display-system}
A {\bf display system}, ${\cal D}$ for a category with a transverse system ${\cal Q}$ consists of a family of ${\cal Q}$-quarrable maps such that:
\begin{itemize}
	\item All isomorphisms are in ${\cal D}$.  
	\item $\cal D$ is closed to pullbacks along arbitrary maps.  
	\item If $q$ and $q'$ are in ${\cal Q}$ and $\alpha: q \to q'$ is a natural transformation between transverses such that for any $(i_1, i_2, \ldots i_n) \ne (0,0,\ldots 0)$, $\alpha_{(i_1,i_2, \ldots i_n)}$ is in $\cal D$ then $\alpha_{(0,0,\ldots 0)}$ is also in $\cal D$.
\end{itemize}
\end{definition}
The significance of the last rather technical condition will become apparent shortly, in proposition \ref{displayed-bun}.  It essentially says that pullbacks of display maps -- as objects of the arrow category -- must also be display maps.

\begin{definition}
A {\bf display} tangent category is a transverse tangent category equipped with a display system on that transverse system such that $T$ preserves display maps.  
\end{definition}

It is important to note that we do not necessarily require that the projection maps $p_M: TM \to M$ be a display map in a display tangent category.  This is important not only for some examples we would like to consider, but also for the development of fibrations of tangent categories (see the note after \ref{tangent-substitution}).  

Nonetheless, when the projection maps \emph{are} in the transverse system, then the transverse system generated by the display maps alone will already include all the transverses required of a tangent category.   Thus, in this case, requiring preservation of the display maps will imply all the preservation requirements of the transverse system.  

\begin{example}  {\em ~
\begin{enumerate}[{\em (i)}]
\item In the tangent category of smooth manifolds, a natural choice of display maps is the submersions.  Note that the projections $p:T(M) \to M$ are submersions, so in this case the projections $p$ are included in the display system.  
\item
Any Cartesian tangent category has a display system generated by projections from products and the maps to the final object. 

Explicitly, we say that a map $f: X \to A$ is a \emph{projection} if there exists a map $f': X \to A'$ such that the pair $(f,f')$ makes $X$ the product of $A$ and $A'$.  In a Cartesian tangent category, we can form a display system in which the display maps are the projections in this sense and the transverses required to make these maps quarrable are all the multiple pullbacks in which all but one leg are projections.  The only axiom that takes a bit of work to check is the naturality axiom.  Suppose we have a natural transformation $\alpha: P \to Q$ between such transverses:
$$
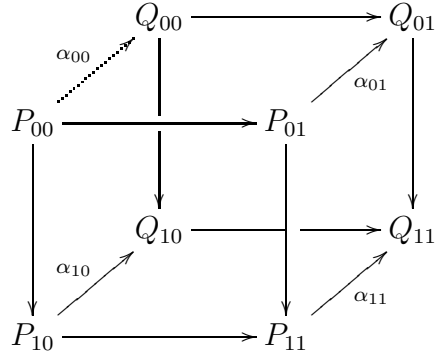
\xymatrix{ & Q_{00} \ar[dd]|{\hole}^<<<<<<{} \ar[rr]^{} & & Q_{01} \ar[dd]^{} \\ 
           P_{00} \ar[dd]_{}  \ar@{..>}[ur]^{\alpha_{00}} \ar[rr] & & P_{01} \ar[dd]^<<<<<<{} \ar[ru]_{\alpha_{01}} \\
           & Q_{10} \ar[rr]_<<<<<<{}|>>>>>>>>>>>>{~~} & & Q_{11} \\
           P_{10} \ar[ur]^{\alpha_{10}} \ar[rr]_{} & & P_{11} \ar[ur]_{\alpha_{11}}}$$
so that each $\alpha_{01}, \alpha_{10}, \alpha_{11}$ are projections, as well as the maps $P_{01} \to P_{11}$ and $Q_{01} \to Q_{11}$.  Then if we have $P_{01} = Q_{01} \times P_{01}'$ and $P_{10} = Q_{10} \times P_{01}'$ then one can easily check that $P_{00}$ is the product of $Q_{00}$, $P_{01}'$, and $P_{01}'$, with $\alpha_{00}$ one of the projections.

Note that in this display system the projections from the tangent bundle need not be display maps.
\item
Every Cartesian differential category is an example of a Cartesian tangent category, as was discussed in Section \ref{diff-structure}, thus the projections form a natural display system.  In this case the 
projections from the tangent bundle {\em are} display maps.  
\item In a model of SDG, all pullbacks exist and are preserved by the tangent functor, so the standard display system on a model of SDG consists of all maps in the category.  
\end{enumerate} }
\end{example}

\begin{definition}
A {\bf display} functor between categories with a display system: 
$$F: (\X,{\cal D}) \to (\X',{\cal D}')$$  
is a transverse functor that sends elements of ${\cal D}$ to elements of ${\cal D}'$.  
\end{definition}
In particular, this means explicitly that if $d \in {\cal D}$ then $F(d) \in {\cal D}'$ and, furthermore, whenever 
$$\xymatrix{D \ar[d]_{d'} \ar[r]^{f'} & C \ar[d]^d \\
                    A \ar[r]_f & B}$$
is a pullback then
$$\xymatrix{F(D) \ar[d]_{F(d')} \ar[r]^{F(f')} & F(C) \ar[d]^{F(d)} \\
                    F(A) \ar[r]_{F(f)} & F(B) }$$
 is a pullback.  
 
\begin{example} {\em ~
As noted above, well-adapted models of SDG consist of, in particular, a transverse functor from smooth models to a model of SDG; such a functor is in addition always a display functor relative to the natural display structure on the model of SDG (which consists of all maps in the category). }
\end{example}
 


\subsection{Morphisms of tangent categories revisited}

In this section we revisit morphisms of tangent categories which, in addition, may have a transverse and/or display system associated to them.  We also consider comorphisms of tangent categories, which were not considered in \cite{sman3}.  

\begin{definition}
A \textbf{morphism of tangent categories with transverse systems} 
	\[ F:(\X,\T, {\cal Q}) \to (\X',\T', {\cal Q}') \] 
consists of a transverse functor $F: \X \to \X'$ and a natural transformation $\alpha: TF \to FT'$ such that the following diagrams commute:
$$\xymatrix{TF \ar[dr]_{pF} \ar[r]^\alpha & FT' \ar[d]^{Fp'} \\ & F} 
~~~~~\xymatrix{F \ar[d]_{0F} \ar[dr]^{F0'} \\ TM \ar[r]_\alpha & FT'}
~~~~~\xymatrix{T_2F \ar[d]_{+F} \ar[r]^{\alpha_2} & FT'_2 \ar[d]^{F+'} \\ TF \ar[r]_\alpha & FT'}$$
$$\xymatrix{TF \ar[d]_\ell \ar[r]^\alpha & FT' \ar[d]^{F\ell'} \\
               T^2F \ar[r]_{(T\alpha) (\alpha T')} & F{T'}^2} ~~~~~~
\xymatrix{T^2F \ar[d]_c \ar[r]^{(T \alpha) (\alpha T')} & F{T'}^2 \ar[d]^{Fc'} \\
               T^2F \ar[r]_{(T \alpha) (\alpha T')} & F{T'}^2}$$
If the categories also have display systems, a \textbf{morphism of display tangent categories} simply asks in addition that $F$ be a display functor.  
\end{definition}


The other possibility is to consider \textbf{comorphisms} between tangent categories: these consist of a functor $F$ together with a transformation $\beta: FT \to T'F$ satisfying the same requirements as above but with the functor order reversed; if the tangent categories have display systems the functor is again required to preserve these.  It is worth noticing that to make sense of comorphisms, the assumption that $F$ is transverse is crucial. Consider the coherence diagram for the addition:
$$\xymatrix{F T'_2 \ar[d]_{F+'} \ar[r]^{\beta_2} & T_2F \ar[d]^{+F} \\ FT' \ar[r]_\beta & TF}$$
To obtain a map $\beta_2: T'_2(F(M)) \to F(T_2(M))$ one needs the codomain to be a limit.

\begin{definition}
We say a morphism of tangent categories is \textbf{strong} if $\alpha$ is an isomorphism and \textbf{strict} if $\alpha$ is the identity.   
\end{definition}

Note that a strong morphism is also a comorphism.    

\begin{definition}  A morphism of transverse tangent categories is {\bf Cartesian} if for each $f$, the naturality square   
$$\xymatrix{F(T(E)) \ar[r]^{\alpha_E} \ar[d]_{F(T(f))} & T'(F(E)) \ar[d]^{T'(F(f))} \\
                    F(T(M)) \ar[r]_{\alpha_{M}} & T'(F(M))}$$ 
is a pullback which is transverse in the codomain category $\X'$.  
\end{definition}

Clearly strong (and strict) morphisms are always cartesian, and the identity functor is a strict morphism of tangent categories.  An important example of a strong morphism is $T$ itself:

\begin{lemma}
If $\X$ is a (display) tangent category, then the pair $(T,c): \X \to \X$ is a strong morphism of tangent categories.
\end{lemma}

\proof
For $T_1=T$ each of the equations required to be a morphism of tangent categories is actually one of the axioms for a tangent category - for example, the last two equations for morphisms of tangent categories require that $cT(\ell) = \ell T(c)c$ and $T(c)cT(c) = cT(c)c$.  As $c$ is invertible and $T$ preserves the necessary pullbacks, it is a strong morphism of tangent structure.   
\endproof

It is straightforward to show that all these morphisms compose in the obvious way; thus, each $T^n$ is also a morphism of tangent categories. 

\begin{definition}
A {\bf transformation} between morphisms of tangent categories $\gamma: F \to G$ is a natural transformation $\gamma$ such that 
$$\xymatrix{T'(F(M)) \ar[d]^{\alpha^F_M} \ar[rr]^{T'(\gamma_M)} & & T'(G(M)) \ar[d]^{\alpha^G_M} \\ 
                    F(T(M)) \ar[rr]_{\gamma_{T(M)}} & & G(T(M))}$$
commutes.   
\end{definition}
A basic example of such a transformation is $p: T \to {\sf Id}$.  The required diagram is just $cT(p) = p$.   It is now straightforward to check:

\begin{lemma}
Multiple transverse pullbacks of (strong, cartesian, co-) morphisms of tangent categories are (respectively, strong, cartesian,co-) morphisms of tangent categories.
\end{lemma}

An immediate corollary of this is:

\begin{corollary} 
Each  $(T_n,c_n)$ is a strong tangent morphism.
\end{corollary}

\proof
Observe that the identity functor is certainly a strong morphism and that $p: T \to {\sf Id}$ clearly has wide pullbacks over itself transverse. The unique map 
$$(c_n)_M : T^2(M) \x_{T(M)} ... \x_{T(M)} T^2(M)  \to T(T(M) \x_M ... \x_M T(M))$$ 
which applies $c$ on each coordinate is certainly an isomorphism.
\endproof

Further important examples of a strong morphism of tangent categories are the product and final functors for Cartesian differential categories  
$$\_\x\_: \X \x \X \to \X ~~~~~~1:  {\mathbf 1}  \to \X$$
saying that the product is preserved by the tangent structure, of course, also means the tangent is preserved by the product.


\subsection{Differential bundles revisited}

Having introduced transverse systems and display system it is necessary to revisit the definition of a differential bundle with a view to understanding how these 
structures interact with that definition.  The key point is that all the pullback diagrams mentioned in the definition must now be transverse:

\begin{definition}
A \textbf{differential bundle} in a tangent category $(\X,\T)$ with a transverse system $\cal Q$ consists of an additive bundle on a map $q$ together with a lift map $\lambda$:
$${\mb q} := (q: E \to M, \sigma: E_2 \to E, \zeta: M \to E, \lambda: E \to T(E))$$ 
such that
\begin{itemize}
        \item Finite multiple pullbacks of $q$ along itself are transverse.
        \item $(\lambda,0): (E, q,\sigma,\zeta) \to (TE, T(q), T(\sigma), T(\zeta))$ is an additive bundle morphism.
	\item $(\lambda,\zeta): (E, q,\sigma,\zeta) \to (TE, p,+,0)$ is an additive bundle morphism.
	\item The {\bf universality of the lift}: 
	The following diagram is a transverse pullback:
	       $$\xymatrix{E_2 \ar[d]_{\p_0q=\pi_1q} \ar[rr]^{\mu} &  & T(E) \ar[d]^{T(q)} \\ M \ar[rr]_{0} & & T(M)}$$
	       where $\mu: E_2 \to TE$ is defined by $\mu := \<\p_0\lambda, \p_10\>T(\sigma)$.
	\item The equation $\lambda \ell_E = \lambda T(\lambda)$ holds.
\end{itemize}
If $(\X,\T)$ is a display tangent category, $\mb{q}$ is a {\bf display} differential bundle if $q$ is a display map.

Morphisms of differential bundles and linear morphisms of differential bundles are defined as in the original definition (\ref{defnDiffBundles}).  

\end{definition}

Note that in a display tangent category, we do not demand that the projection $q: E \to M$ of a differential bundle be a display map.  One argument for avoiding this demand is because we should like $p:T(M) \to M$ to be a differential bundle, and for various reasons we have not demanded that $p$ be a display map.  That said, when $q$ is a display map, that is when we have a \emph{display} differential bundle, then the definition of such a differential bundle need not include the requirements that multiple pullbacks of $q$ exist and are transverse, nor that the diagrams associated to the universality of the lift are transverse: all these will be automatic.   Furthermore display differential bundles can always be pulled back.   Thus, these bundles have particularly nice properties and indeed this is part of the point of much of the work that follows.  



An important observation, which allows the construction of many examples of differential bundles, is:

\begin{proposition}\label{preservationOfDiffBundles}
If $(F,\alpha): \X \to \X'$ is a Cartesian morphism of transverse tangent categories then there is a functor 
$${\sf DBun}(F): {\sf DBun}(\X) \to {\sf DBun}(\X'); $$
$$ ~~~~~\mb{q} \mapsto  (Fq,F(\zeta)\alpha,F(\sigma)\alpha, F(\lambda)\alpha)$$ 
Furthermore, when $(f,g): \mb{q} \to \mb{q'}$ is a linear bundle morphism then ${\sf DBun}(F)(f,g) = (F(f),F(g))$ is a linear bundle morphism.
\end{proposition}
\proof
The equational axioms for $F(\mb{q})$ to be a differential bundle and for $(Ff,Fg)$ to be a linear bundle morphism follow in a straightforward fashion from the equations for $(F,\alpha)$ being a morphism of tangent categories and naturality of $\alpha$.  The only slight difficulty is the universal of the lift for $F(\mb{q})$; this asks that
\[
\bfig
	\square<500,350>[(FE)_2`T'FE`FM`T'FM;\mu'`\pi_0 F(q)`T'(Fq)`0']
\efig
\]
be a pullback; however, we may re-express this diagram as the outer square of the composite
\[
\bfig
	\square<500,350>[(FE)_2`F(E_2)`FM`FM;`\pi_0F(q)`F(\pi_0 q)`1]
	\square(500,0)<500,350>[F(E_2)`FTE`FM`FTM;F\mu``F(Tq)`F(0)]
	\square(1000,0)<500,350>[FTE`T'FE`FTM`T'FM;\alpha``T'(Fq)`\alpha]
\efig
\]
the leftmost square has isomorphisms as the top arrows, the middle square is a pullback since $F$ preserves the pullback diagram for $T(\mb{q})$, and the rightmost square is a pullback (and transverse) as the morphism is cartesian.  Thus the entire square is a pullback, as required.  
\endproof

The following result relates the bracketing operation of a differential bundle $\sf q$ to the bracketing operation of the differential bundle  $F({\sf q})$ (for $F$ a Cartesian morphism).  

\begin{lemma}\label{lemmaBracketT}
Suppose that ${\sf q}$ is a differential bundle, $f: X \to TE$ equalizes $T(q)$ and $pq0$ and $(F,\alpha): \X \to \Y$ is a Cartesian morphism of tangent categories.   Then $F(\{f\}) = \{F(f)\alpha\}$.
\end{lemma}
\begin{proof}
We first need to show $\{F(f)\alpha\}$ is well-defined; that is, we need to show that 
	\[ F(f)\alpha T(F(q)) = F(f)\alpha T(F(q))p0. \]
Indeed,
\begin{eqnarray*}
F(f)\alpha T(F(q) )
& = & F(f)F(T(q))\alpha \mbox{ (naturality of $\alpha$)} \\
& = & F(fT(q))\alpha \\
& = & F(fT(q)p0)\alpha \mbox{ (assumption on $f$)} \\
& = & F(f)F(T(q))F(p)F(0)\alpha \\
& = & F(f)F(T(q))\alpha p0 \mbox{ (coherences on $\alpha$)} \\
& = & F(f)\alpha T(F(q))p0 \mbox{ (naturality of $c$)}
\end{eqnarray*}
We now check $F(\{f\})$ has the same universal property as $\{T(f)\alpha\}$:
\begin{eqnarray*}
\lefteqn{\<F(\{f\})F(\lambda)\alpha, F(f)\alpha p0\>T(F(\sigma)) }\\
& = & \<F(\{f\}\lambda)\alpha,F(f)F(p)F(0)\alpha \> T(F(\sigma)) \mbox{ (coherence for $\alpha$)} \\
& = & \<F(\{f\}\lambda),F(fp0) \> F(T(\sigma))\alpha \mbox{ (naturality of $\alpha$)} \\
& = & F(\<\{f\}\lambda,fp0\>T(\sigma))\alpha \\
& = & F(f) \alpha
\end{eqnarray*}
as required. 
\end{proof}



As the product functor for cartesian differential categories is a strong morphism, it follows that the product of two differential bundles is also a differential bundle.


\section{Tangent fibrations and differential fibrations}\label{secTanFibrations}

 \subsection{Tangent fibrations}

This paper's main contribution so far is the abstract formulation of differential bundles for tangent categories.  This section describes some basic aspects of the general theory of tangent fibrations in order to place this development in a wider context.     It will be assumed for this section that the reader is familiar with the general theory of fibrations.  For the reader less familiar with these notions we recommend \cite{bartjacobs} or \cite{borceaux2}.  

Before we start it is necessary to make some remarks about the closure of transverse systems in fibrations.  Suppose we have a fibration $\partial: (\X,{\cal Q}) \to (\B,{\cal Q}')$ between categories with transverse systems such that $\partial$ is a transverse functor.  Suppose also that $q: {\mathbf 2}^n \to \X$ is transverse in $\X$, $q': {\mathbf 2}^n \to \B$ is transverse in $\B$, and $\alpha: q' \Rightarrow q \partial $ is a natural transformation of transverses.  From this data we can define a functor $\alpha^{*}(q): {\mathbf 2}^n \to \X$ with $\alpha^{*}(q)\partial = q'$: each vertex $v \in {\mathbf 2}^n$ determines the vertex $\alpha^{*}(q)(v)$ by setting it equal to $\alpha_v^{*}(q(v))$, and the maps between the vertexes are then determined by requiring that they sit above the corresponding map of $q'$.

By construction $\alpha^{*}(q)$ is a functor and preserves all the limits that $q$, $q\partial$, and $q'$ preserve.  (For transverses it suffices to show that pullbacks will be preserved by condition (d) of the definition of a transverse system.)  

\begin{definition} 
With the definitions above, we call $\alpha^*(q)$ a {\bf substituted transverse}. 
\end{definition}

We will use these substituted transverses repeatedly in the sequel.  It is  easy to see that if $(F,F'): \partial \to \partial'$ is a morphism of fibrations in which all re-indexing functors are transverse then this will still be the case if we add the substituted transverses to the transverse system in $\X$.  Thus, in a fibrational setting it makes sense  to insist that the transverse system of the total category include substituted transverses.  Accordingly, we make the following definitions.

\begin{definition}  Suppose $\partial: \X \to \B$ is a fibration.
\begin{itemize}
	\item If $(\X,{\cal Q})$ and $(\B,{\cal Q}')$ are transverse systems, then $\partial$ is a \textbf{transverse fibration} if $\partial$ is a transverse functor and ${\cal Q}$ includes all subsituted transverses.
	\item If, in addition, $(\X,{\cal Q})$ and $(\B,{\cal Q}')$ have display systems ${\cal D}$ and ${\cal D'}$, then $\partial$ is a \textbf{display fibration} if $\partial$ is a display functor.
	\item If $(\X,\T)$ and $(\B,\T')$ are tangent categories, then $\partial$ is a \textbf{tangent fibration} if $\partial$ is a strict morphism of tangent categories and $(T,T')$ is a morphism of fibrations.
	\item A \textbf{display tangent fibration} is simply a tangent fibration which is also a display fibration.
\end{itemize}
\end{definition}

A consequence of $(T,T')$ being a morphism of fibrations is that the functors $(T_n,T'_n): \partial \to \partial$ are also morphisms of fibrations: the transformations associated with the tangent structure all then become fibred transformations between these morphisms of fibrations.  

One of the major results of this section is the following:
 
 \begin{theorem} \label{tangent-substitution}
In any (display) tangent fibration
$$\partial: (\X,\T) \to (\B,\T')$$ 
\begin{itemize}
	\item each fibre $\partial^{-1}(M)$ can be given the structure of a (display) tangent category; we call this the \textbf{vertical tangent structure} on $\partial^{-1}(M)$, 
	\item the substitution functors $h^*: \partial^{-1}(M) \to \partial^{-1}(M')$ are strong (display) tangent morphisms.
\end{itemize}
 \end{theorem}
\begin{proof}
We must first define the tangent structure $\T_M$ on a fibre $\partial^{-1}(M)$.  If  $\partial(E)=M$, define $T_M(E):= (0_M)^{*}(T(E))$.  Define $(p_M)_E :=  0^{*}_{T(E)}p_E$.  Note that this can be regarded as the projection $p_E: T(E) \to E$ substituted along the cone 
 $$\xymatrix@R=10pt{T_M(E) \ar@{..>}[dd]_{(p_M)_E} \ar[rr]^{0^{*}_{T(E)} } & & T(E) \ar[dd]_{p_E} \\ \\
                      E \ar@{=}[rr]  & ~ \ar[d]_{\partial} & E \\
 & ~ & T'(M) \ar[dd]^{p_M} \\ M \ar[urr]^{0}  \ar@{=}[drr] \\  & & M} $$
 The fact that $T$ is functorial and  $p$ is natural immediately shows that $T_M$ is a functor.  The fact that we are substituting in this manner immediately means that, in the fibre over $M$, wide pullbacks of $(p_M)_E$ are in the transverse system, thus $T$ preserves them, and from this one can easily show $T_M$ preserves them.  The remaining required transformations for tangent structure can be defined similarly by substitution over the cones:
 $$\xymatrix@R=10pt{&& T'(M) \ar[dd]^{\ell_M} \\ M \ar[rru]^{0} \ar[rrd]_{0T'(0)} \\ & & T'^2(M)}~~~~
     \xymatrix@R=10pt{ & &  T'^2(M) \ar[dd]^{c_M} \\ M \ar[rru]^{0T'(0)} \ar[rrd]_{0T'(0)} \\ & & T'^2(M)}$$
$$\xymatrix@R=10pt{ & & M \ar[dd]^0 \\ M \ar@{=}[rru] \ar[drr]_0 \\ & & T'(M)} ~~~~~ 
     \xymatrix@R=10pt{ & & T'_2(M) \ar[dd]^{+} \\ M \ar[rru]^{0_2} \ar[rrd]_{0} \\ & & T'(M)}$$
These use the fact that the zero maps provide a fixed point for the transformations in tangent structure.   This means all the coherence diagram can then be transmitted into the local fibre by substituting over these pointed cones and this also ensures that all the transverse (i.e. limit) information is preserved.  This implies immediately that the fibre with respect to this structure is a tangent category.

To show that the substitution functors are strong morphisms of the tangent structure it suffices to observe that all the substitution cones are natural in $M$ and thus transporting the substitutions along an $f$ will result in structure which is equivalent up to a unique vertical isomorphism.

To complete the proof of the theorem we must also discuss the display structure.  We define the display system of the fibre to simply be the vertical display maps: ${\cal D}_M := {\cal D} \cap \partial^{-1}(M)$.   Consider pulling back a vertical display map against a vertical map in $\X$.  Its pullback exists and is still a display map.  Moreover, by assumption, the pulback is preserved by $\partial$.  But in $\X'$ the pullback becomes a pullback of identity maps, so the pullback is itself vertical and so is contained in the fibre.

Now, we need to show that each substitution functor preserves these display maps and that each $T_M$ preserves the display maps.  For both of these, it is useful to first see that if $f: A \to B$ is a vertical map then 
 $$\xymatrix{ h^{*}(A) \ar@{..>}[d]_{h^{*}(f)} \ar[rr]^{h_A^{*}}    & & A \ar[d]^{f} \\
                      h^{*}(B) \ar[rr]_{h_B^{*}} & & B}$$
is a pullback.  Indeed, suppose there are $z_1: Z \to A$ and $z_2: Z \to h^{*}(B)$ such that $z_2h_B^{*} = z_1f$.  Then since $h_B^*$ is vertical, $\partial(z_1) = \partial(z_2)h$.  But then since $\partial$ is a fibration there is a unique map $k: Z \to h^*(A)$ with $kh^{*}_A=z_1$ and $\partial(k) = \partial(z_2)$.  But then $\partial(kh^{*}(f)) = \partial(z_2)$ since $h^{*}(f)$ is vertical, and
	\[ kh^{*}(f) h^{*}_B = k h^*_Af = z_1f = z_2 h^{*}_B \] 
so $kh^{*}(f)=z_2$.  This proves it is a pullback.

Then if $f$ is a vertical display map then $h^*(f)$ is the pullback of a display map and hence is itself a display map, so each subsitution functor is preserves display maps.

Finally, as $T$ preserves display maps and $T_M(f)$ is given by
 $$\xymatrix{T_M(A) \ar@{..>}[d]_{T_M(f)} \ar[rr]^{0^{*}_A} & & T(A) \ar[d]_{T(f)} \\
                     T_M(B) \ar[rr]_{0^{*}_B} & & T(B)}$$
then as $T_M(f)$ is the pullback of a display map it is also a display map.  So $T_M$ also preserves display maps, as required.  
\end{proof}
\textbf{Note}: the projections $p_M$ in the fibres need not be display maps, even if the projections in $\X$ and $\B$ are.

\begin{definition}
A (display) tangent fibration $\partial: \X \to \B$  is  a {\bf tangent bifibration} in case $\partial$ is a cofibration as well and cosubstitutions of transverses (display maps) are transverse (display maps).  
\end{definition}
If $\alpha: q'\partial \Rightarrow q$  is a a  natural transformation between transverses in $\B$ then bifibrational structure allows us to define  $\exists_\alpha(q')$ by dualizing the argument for substituting transverses.  Thus if $v$ is a vertex of ${\mathbf 2}^n$ we define $\exists_\alpha(q')(v)$ to be $\exists_{\alpha_v}(q'(v))$  and then the maps between these vertices are determined by requiring them to sit above those of $q$.  This certainly gives a functor $\exists_\alpha(q'): {\mathbf 2}^n \to \X$; however, there is no reason why it should preserve pullbacks or be an existing transverse.  The non-trivial requirement of being a tangent bifibration is that this cosubstitution is already a transverse (or a display) functor.
 
 Given a tangent bifibration there is a second way to induce tangent structure onto the fibres.  This is by cosubstituting the final transformations of the tangent structure given by the projections:
 $$\xymatrix@R=10pt{T(E) \ar[dd]_{p_E} \ar[rr]^{\exists_p^{T(E)} } & & \exists_p(T(E)) =T^M(E) \ar@{..}[dd]_{\widetilde{p_E}} \\ \\
                      E \ar@{=}[rr]  & ~ \ar[d]_{\partial} & E \\
                      T'(M) \ar[dd]_{p_M} \ar[rrd]^p  & ~ &\\ & & M \\ M \ar@{=}[rru] }$$
 The fact that $T$ is functorial and $p$ is natural immediately shows that $T^M$ is a functor.  The condition on cosubstituting means that, in the fibre over $M$, wide pullbacks of $(p^M)_E$ are in the transverse system, thus $T$ preserves them so easily $T^M$ preserves them (and similarly for display maps).  The remaining required transformations for tangent structure can be defined by cosubstitution over the cocones:
 $$\xymatrix@R=10pt{T'(M) \ar[dd]_{\ell_M} \ar[drr]^{p} \\ & & M \\ T'^2(M)  \ar[rru]^{p}}~~~~
     \xymatrix@R=10pt{T'^2(M) \ar[dd]_{c_M} \ar[rrd]^{pT'(p)} \\ & & M  \\ T'^2(M) \ar[rru]_{pT'(p)}}$$
$$\xymatrix@R=10pt{ M \ar[dd]^0 \ar@{=}[drr] \\ & & M  \\ T'(M) \ar[rru]_{p} } ~~~~~ 
     \xymatrix@R=10pt{T'_2(M) \ar[dd]^{+} \ar[rrd]^{\pi_0p=\pi_1p} \\  & & M   \\ T'(M) \ar[rru]_{p}}$$
this make $\T^M$ an alternative tangent structure on the fibres which we call the {\bf total} tangent structure.  This gives:

\begin{proposition}
In a (display) tangent bifibration the fibres have two induced tangent structures, the vertical tangent structure and the total tangent structure and there is a morphism 
of tangent structures given by the identity functor together with the natural transformation 
$$T_M(E) \to^{0^{*}} T(E) \to^{\exists_p} T^M(E).$$
\end{proposition}

Notice that the fact that this is a morphism of tangent structures (which trivially is transverse and preserves display maps) is immediate from its definition. This, therefore gives us not only an example of a morphism of tangent structure which, in general, is not strong or Cartesian, but also an additional source of examples of categories with more than one tangent structure.

\begin{example}\label{tanFibrationExamples}{\em ~
\begin{enumerate}[{\em (i)}]
\item When $\X$ is a Cartesian differential category the simple fibration $\partial: S[\X] \to \X$ is a tangent fibration.  Recall that the category $S[\X]$ has object pairs of objects of $\X$, $(A,X)$, and a map $(f,h): (A,X) \to (B,Y)$ is a pair of maps $f: A \to B$ and $g: A \x X \to Y$ (here $A$ is regarded as a ``context'') with composition $(f,g)(f'g') = (ff', \<\pi_0f,g\>g')$ and identity $(1_A,\pi_1)$.   The functor $\partial$ 
simply picks the first coordinate.   It is not hard to see that this is a Cartesian differential category with $D(f,g):= (D(f),{\sf ex}D(g)): (A \x A,X \x X) \to (B,Y)$, where 
${\sf ex}: A \x A \x X \x X \to A \x X \x A \x X$ swaps the middle two coordinates.   The vertical tangent functor is obtained by substituting along the zero map: in the fibre over $A$ this gives the 
functor 
$$T_A: \partial^{-1}(A) \to \partial^{-1}(A); \begin{array}[c]{lcl} \xymatrix{ (A,X) \ar[d]_{(1,g)} \\ (A,Y)} \end{array} ~~ \mapsto ~~ 
                     \begin{array}[c]{lcl} \xymatrix{ (A,X \x X) \ar[d]^{(1,(\< 0,1\> \x 1){\sf ex}D(g))} \\ (A,Y \x Y)} \end{array}$$
This is precisely the ``partial derivative'' with respect to $X$ described in \cite{cartDiff}.
\item Consider a tangent category, $\X$, such as one arising from SDG, in which all maps can be viewed as display maps.  A bundle over $M$ is just a map $q: E \to M$ and the category of bundles, also known as the standard fibration,  $\partial_1: {\sf bun}(\X) = \X^{\mathbf 2} \to \X$ is also a tangent bifibration.  The tangent structure on $\X^{\mathbf 2}$ is given ``pointwise'' by $qp: qT \to q$ where $q: {\mathbf 2} \to \X$ (see Example \ref{tangent-category-examples} (iv)).  This gives two tangent structures on each slice category of a tangent category as noted in \cite{rosicky}, pages 4-5.  
\end{enumerate}
}
\end{example}

In an arbitrary tangent category, however, the standard fibration $\partial_1: {\sf bun}(\X) = \X^{\mathbf 2} \to \X$ is not a fibration as not all pullbacks need exist, and even if they do, it need not be a tangent fibration as $T$ need not preserve these pullbacks.  This is indeed one important reason for defining transverse and display systems for a tangent category.  

If we do have a display tangent category, then we can let ${\sf bun}_{\cal D}(\X)$ denote the full subcategory of the arrow category, ${\sf bun}(\X) = \X^{\mathbf 2}$ consisting of just display maps.  It is  well-known that this will result in a fibration $P: {\sf bun}_{\cal D}(\X)  \to \X$; however, we must also 
supply a transverse and display system. There are two possible canonical choices. The first, and more 
restrictive choice, insists that the display maps and maps involved in the transverse system are contained in the Cartesian maps (i.e. the systems 
only involve morphisms between bundles which are already pullback squares).  The second choice -- and the choice on which we will focus -- is to allow as 
transverse any cube whose base and total components are transverse in the original system and, similarly, as a display map any map whose base and 
total map is a display map.  That this indeed gives a display tangent category requires the use of the third bullet of Definition \ref{display-system} to ensure that 
pullbacks of display maps are display maps.

\begin{proposition} \label{displayed-bun}
If $\X$ is a display tangent category, then so is ${\sf bun}_{\cal D}(\X)$ (with the above display system) and the projection functor $P: {\sf bun}_{\cal D}(\X) \to \X$ is a tangent fibration.  
\end{proposition}

\begin{proof}
This is mostly straightforward. For example, $T$ sends
\[
\bfig
\square<500,250>[E_1`E_2`M_1`M_2;f`q_1`q_2`g]
\place(750,125)[\mbox{to}]
\square(1000,0)<500,250>[TE_1`TE_2`TM_1`TM_2;Tf`Tq_1`Tq_2`Tg]
\efig
\]
and the projection of $Tq: TE \to TM$ to $q: E \to M$ is simply the pair $(p_E,p_M)$.  The other structural transformations are similarly defined, and as pullbacks in the arrow category are defined pointwise, all the axioms are immediate.  The fact that pullbacks of display maps are again display maps in ${\sf bun}_{\cal D}(\X)$ is a direct consequence of the third axiom for a display system.  
\end{proof}

We would like to restrict this fibration even further, from arbitrary bundles to differential bundles.  The fact that $T(\mb{q})$ is a differential bundle and linear bundle morphisms are preserved by $T$ (see proposition \ref{preservationOfDiffBundles}) means ${\sf DBun}(\X)$ is a tangent category as it is a full subcategory of $\X^{\mathbf 2}$ which is closed to the tangent structure.  Furthermore, ${\sf DBun}(\X)_{\sf Lin}$ is a tangent category: this requires checking that the vertical lift and canonical flip are linear morphisms which is straightforward.  Thus, ${\sf DBun}_{\cal D}(\X)$ and ${\sf DBun}_{\cal D}(\X)_{\sf Lin}$ are tangent categories as they are subcategories of $\X^{\mathbf 2}$ which are closed to the tangent structure.  This implies:

\begin{corollary}
$P: {\sf DBun}(\X) \to \X$ and $P: {\sf DBun}(\X)_{\sf Lin} \to \X$ are strong tangent functors.  Similarly $P: {\sf DBun}_{\cal D}(\X) \to \X$ and $P: {\sf DBun}_{\cal D}(\X)_{\sf Lin} \to \X$ are strong tangent functors.
\end{corollary}

For the display and transverse structure of ${\sf DBun}(\X)$ we have to be a little more careful because it is not the case that a pullback of differential bundles in the bundle category ${\sf bun}_{\cal D}$ will again be a differential bundle (however, as noted above, it will certainly be a display map).  When we restrict to linear morphisms, however, it is 
not hard to check that the pullback \emph{will} be a differential bundle due to the compatibility linear maps have with the lifts of the differential bundles.   Thus, in defining the transverse and
display system for ${\sf DBun}_{\cal D}(\X)$, we must restrict the transverse and display systems of ${\sf bun}_{\cal D}(\X)$ to lie within the linear maps.  With this caveat we have:

\begin{corollary} \label{displayed-bundles}
For any display tangent category, $\X$,  the categories of differential bundles ${\sf DBun}_{\cal D}(\X)$ and ${\sf DBun}_{\cal D}(\X)_{\sf Lin}$, with the transverse and display system indicated above, are  display tangent categories.  

Furthermore, $P: {\sf DBun}_{\cal D}(\X) \to \X$ and $P:{\sf DBun}_{\cal D}(\X)_{\sf Lin} \to \X$  are tangent fibrations and, moreover, each fibre has finite products given by the Whitney sum.
\end{corollary}

We provide an alternative view of the Whitney sum (see the remarks at the end of section \ref{secConstDiffBundles}) of two differential bundles $q: E \to M$ and $q': E' \to M$ with the same base whose projections are display maps.  First observe that  $E \x_M E' \to^{\pi_0q} M$ certainly has all pullback powers along itself, which are always transverse, as 
this is a display map. Furthermore, being the product of $E \to^q M$ and $E' \to^{q'} M$ in the slice makes this  immediately an additive bundle. 
To see that it is a differential bundle we exhibit a lift map $\lambda_2$ as follows:
$$\xymatrix{ & E \x_M E' \ar[ddl]_{\pi_0} \ar[dr]_{\pi_1} \ar@{..>}[rr]^{\lambda_2} & & T(E \x_M E') \ar[ddl]_<<<<<<<<{T(\pi_0)}|{\hole} \ar[dr]^{T(\pi_1)} \\
                     & & E' \ar[ddl]^<<<<<<{q'} \ar[rr]_{\lambda'} & & T(E') \ar[ddl]^{T(q')} \\
                     E \ar[dr]_{q} \ar[rr]^{\lambda}|>>>>>>>\hole & & T(E) \ar[dr]_{T(q)} \\
                     & M \ar[rr]_{0}  & &T(M)}$$
It is then clear that $\lambda_2$ is a lift map and will have the required additive bundle morphism properties.

\begin{definition}
A tangent fibration is {\bf differential} in case all its fibres with their vertical tangent structure have coherent differential structure (see definition \ref{defnCoherentDiffStructure}).  
\end{definition}
The simple fibration over a Cartesian differential category (see \ref{tanFibrationExamples}.i) is clearly an example of a differential tangent fibration.  In the final section of this paper we will show that the tangent fibration of display differential bundles is also a differential tangent fibration.


\subsection{The tangent fibration of display differential bundles}\label{diff-fibrations}

A rather unsatisfactory aspect of the proof that the display differential bundles form a tangent fibration,  
$P: {\sf DBun}_{\cal D}(\X) \to \X$ (Corollary \ref{displayed-bundles}), is that it did not provide a concrete description of the vertical tangent structure in a fibre.   In this section we examine 
the local (vertical) tangent structure in each fibre, ${\sf DBun}_{\cal D}(\X)[M]$, and we will  show that $P: {\sf DBun}_{\cal D}(\X) \to \X$ is a differential tangent fibration.  

Let us start by considering the larger fibration $P: {\sf bun}_{\cal D}(\X) \to X$ described in Proposition \ref{displayed-bun}.  The tangent bundle of the differential bundle ${\sf q}$ is given by $0^{*}({\sf q})$.  This means the projection of the bundle is given by the pullback:
$$\xymatrix{T_M(E) \ar[d]^{p_{\sf q}} \ar[rr]^{0^{*}_{T(E)}} && T(E) \ar[d]^{T(q)} \\
                   M \ar[rr]_0 && T(M) }$$

The vertical lift for this tangent structure is defined as the dotted arrow in:
$$\xymatrix{T_M(E) \ar@{..>}[dr]^{\ell_M} \ar[ddr]_{p_{\sf q}} \ar[rrr]^{0^{*}_{T(E)}} & & & T(E)\ar[rrd]^{\ell} \ar@/_2pc/[ddd]|<<<<<<<<<\hole|>>>>>>>>>>\hole_{T(q)} \\
          & T_M^2(E) \ar[d]^{T_M(p_{\sf q})p_{\sf q}} \ar[rr]^{} & & T(T_ME) \ar[d]^{T(p_{q})T(q)} \ar[rr]^{T(0^*_{TE})} && T^2(E) \ar[d]^{T^2(q)} \\ 
          & M \ar[rr]_0 \ar[rrd]_0 & & T(M) \ar[rr]_{T(0)} & & T^2(M) \\
          & & & T(M) \ar[rru]_{\ell} }$$

The canonical flip is defined as the dotted arrow in:
$$\xymatrix{T_M^2(E)) \ar@{..>}[dr]^{c_M} \ar[ddr]_{T_M(p_{\sf q})p_{\sf q}} \ar[rrrr]^{0^*_{TE}T(0^*_{TE})} & & & & T^2(E) \ar[dl]^c \ar@/^1pc/[ddl]^{T^2(q)c} \\
           & T_M^2(E) \ar[d]^{T_M(p_{\sf q})p_{\sf q}} \ar[rr]^{0T(0))^{*}_{T_2(E)}} & & T^2(E) \ar[d]^{T^2(q)} \\
           & M \ar[rr]_{0T(0)} & & T^2(M)}$$

A slightly surprising observation is that each fibre is actually a {\em Cartesian\/} tangent category even when the original display tangent category 
$\X$ is not:

\begin{lemma} 
For any display tangent category, $\X$, each fibre of the bundle fibration,  ${\sf bun}_{\cal D}(\X)[M]$, is a Cartesian tangent category.
\end{lemma}

\proof
To see this we need to have finite products in the fibre which are preserved by $T_M$.  The final object in the fibre is ${\sf 1}_M$ and each ${\sf q}$ has 
a unique linear morphism $(q,1_M): {\sf q} \to {\sf 1}_M$ and thus the product in the fibre is given by pulling back these morphisms (these
pullbacks always exist as $q$ is a display map).  But $T$ also preserves all these pullbacks as does the substitution functor (as it is given by pulling back).  
Thus, it suffices to verify that $T_M$ preserves the final object.  However, $T_M({\sf 1}_M)$ is given by the pullback:
$$\xymatrix{T_M(M) = M \ar@{=}[d] \ar[rr]^0 & & T(M) \ar[d]^{T(1_M)} \\ M \ar[rr]_{0} & & T(M)}$$ 
and so $T_M$ does indeed preserve the final object.  As the substitution functors are given by pulling back it is now immediate that they preserve 
these products. 
\endproof

We next observe that, for display tangent categories, a differential bundle over $M$ is the same as a differential object in the fibre over $M$ in the 
bundle fibration:

\begin{proposition}\label{displayedDiffBundleEquivalence}
For a display tangent category $\X$ the following are equivalent:
\begin{enumerate}[(i)]
\item A display differential bundle in $\X$ over $M$.
\item A differential bundle over the final object in ${\sf bun}_{\cal D}(\X)[M]$.
\item A differential object in  ${\sf bun}_{\cal D}(\X)[M]$.
\end{enumerate}
\end{proposition}

\begin{proof}
The equivalence between differential objects and differential bundles over the final object for Cartesian tangent categories was established in Proposition \ref{diffObjeqdiffBun}.  Thus it remains only to prove the equivalence of {\em (i)} and {\em (ii)}. 

Given a differential bundle $\mb{q}$ over $M$, we define a differential bundle over $1: M \to M$ in ${\sf bun}_{\cal D}(\X)[M]$ with lift $\lambda'$ given by
\[ \xymatrix{E \ar@/_1pc/[ddr]_{q} \ar@{..>}[dr]^{\lambda'} \ar@/^1pc/[drrr]^{\lambda} \\
             & T_M(E) \ar[d]_{p_{\mb{q}}} \ar[rr]^{0^*} & & TE \ar[d]^{T(q)} \\
             & M \ar[rr]_{0} & & TM } \]
Its projection is simply $q$ (now viewed as a map to $1: M \to M$), and the addition and zero are defined as for $E$.  Conversely, given a $\lambda': E \to T_M(E)$, we define a $\lambda: E \to TE$ by $\lambda = \lambda' 0^*$.  Checking that all axioms are satisfied is now straightforward.   
\end{proof}

\begin{remark}\label{remarkPullbackDiffBundles} {\em For display tangent categories, we can now exhibit a more conceptual proof of the fact that pullbacks of display differential bundles are differential bundles (see Lemma \ref{pullbackDiffBundle}).  By the above, $\mb{q}$ is a display differential bundle if and only if it is a differential bundle over $1$ in ${\sf bun}_{\cal D}(\X)[M](\mb{X})$.  By Theorem \ref{tangent-substitution}, for any $f:  N \to M$, $f^{*}: {\sf bun}_{\cal D}[M] \to {\sf bun}_{\cal D}[N]$ is a strong tangent functor which is easily seen to preserve products.  By proposition \ref{preservationOfDiffBundles}, strong functors carry differential bundles to differential bundles; thus $f^{*}(\mb{q})$ is a differential bundle over $1$ in ${\sf bun}_{\cal D}[N]$ and so $f^{*}(\mb{q})$ is a differential bundle over $N$. }
\end{remark}

Clearly we have ${\sf DBun}_{\cal D}(\X)_{\sf Lin} \subseteq {\sf DBun}_{\cal D}(\X) \to {\sf bun}_{\cal D}(\X)$, and their canonical functors to $\X$ make these morphisms of fibrations.   This means that the objects of ${\sf DBun}_{\cal D}(\X)[M]$ may be viewed as differential objects (or differential bundles over $1$) in ${\sf bun}_{\cal D}(\X)$.  This means that each object 
in  ${\sf DBun}_{\cal D}(\X)[M]$ has a natural assignment of differential structure which we now show is coherent.  

\begin{theorem} \label{display-is-diff-fib}
For any display tangent category $\X$ the fibration 
$$P: {\sf DBun}_{\cal D}(\X) \to \X$$ is a differential tangent fibration.
\end{theorem}
\proof We are required to show that the category ${\sf DBun}_{\cal D}(\X)[M]$ has coherent differential structure.  As noted above, every object in ${\sf DBun}_{\cal D}(\X)[M]$ is a differential bundle over $M$ and thus is a differential bundle over the final object in this category.  In other words, each object has a natural structure as a differential bundle over the final object.   It remains to show that this choice of bundle structure is coherent; that is, we must show {\bf [CDS.1]} and {\bf [CDS.2]} hold.  

Recall that the product in the fibre is the Whitney sum.   An inspection of the definition of the Whitney sum immediately reveals that it is defined using the requirements of {\bf [CDS.1]}!  Thus,   
{\bf [CDS.1]} is satisfied by definition.

For  {\bf [CDS.2]} note that the local vertical tangent bundle of ${\sf q} = (q: E \to M,\sigma,\xi.\lambda)$ is obtained by substitution along $0: M \to T(M)$ of the differential bundle: 
$$T({\sf q}) = (T(q): T(E) \to T(M),T(\sigma),T(\xi),T(\lambda)c).$$
The substituted bundle $0^*(T({\sf q}))$ then has the form  required by {\bf [CDS.2]}.
\endproof

\end{document}